\documentclass[reqno]{amsart}

\usepackage{amssymb}
\usepackage{mathrsfs}
\usepackage{amsmath}
\usepackage{amsthm}
\usepackage{url}
\usepackage{stackrel} 
\usepackage{courier}
\usepackage[all,cmtip]{xy} 
\usepackage{multicol} 
\usepackage{verbatim} 
\usepackage{mathtools}  
\usepackage{hyperref}
\usepackage{multicol}  
\usepackage{enumerate} 
\usepackage{bookmark} 
\usepackage{ytableau}
\usepackage{microtype}
\usepackage[makeroom]{cancel}  
 
\catcode`~=11 
\newcommand{\urltilde}{\kern -.15em\lower .7ex\hbox{~}\kern .04em}  
\catcode`~=13 

\newcommand{\diag}{\mathop{\mathrm{diag}}\nolimits}

\newcommand{\lie}{\mathop{\mathrm{Lie}}\nolimits}

\newcommand{\Ad}{\mathop{\mathrm{Ad}}\nolimits}
\newcommand{\ad}{\mathop{\mathrm{ad}}\nolimits}

\newcommand{\im}{\mathop{\mathrm{Im}}\nolimits}

\newcommand{\In}{\mathop{\mathrm{In}}\nolimits}

\newcommand{\I}{\mathop{\mathrm{I}_n}\nolimits}

\newcommand{\rev}{\mathop{\mathrm{rev}}\nolimits}
\newcommand{\lex}{\mathop{\mathrm{lex}}\nolimits}
\newcommand{\spec}{\mathop{\mathrm{Spec}}\nolimits}  
\newcommand{\wt}{\mathop{\mathrm{wt}\,}\nolimits}
\newcommand{\tr}{\mathop{\mathrm{tr}}\nolimits}

\newcommand{\FHilb}{\mathop{\mathrm{FHilb}}\nolimits}

\newcommand{\Hilb}{\mathop{\mathrm{Hilb}}\nolimits}

\newcommand{\End}{\mathop{\mathrm{End}}\nolimits}

\makeatletter
\newif\if@borderstar
   \def\bordermatrix{\@ifnextchar*{%
       \@borderstartrue\@bordermatrix@i}{\@borderstarfalse\@bordermatrix@i*}%
   }
   \def\@bordermatrix@i*{\@ifnextchar[{\@bordermatrix@ii}{\@bordermatrix@ii[()]}}
   \def\@bordermatrix@ii[#1]#2{%
   \begingroup
     \m@th\@tempdima8.75\p@\setbox\z@\vbox{%
       \def\cr{\crcr\noalign{\kern 2\p@\global\let\cr\endline }}%
       \ialign {$##$\hfil\kern 2\p@\kern\@tempdima & \thinspace %
       \hfil $##$\hfil && \quad\hfil $##$\hfil\crcr\omit\strut %
       \hfil\crcr\noalign{\kern -\baselineskip}#2\crcr\omit %
       \strut\cr}}%
     \setbox\tw@\vbox{\unvcopy\z@\global\setbox\@ne\lastbox}%
     \setbox\tw@\hbox{\unhbox\@ne\unskip\global\setbox\@ne\lastbox}%
     \setbox\tw@\hbox{%
       $\kern\wd\@ne\kern -\@tempdima\left\@firstoftwo#1%
         \if@borderstar\kern2pt\else\kern -\wd\@ne\fi%
       \global\setbox\@ne\vbox{\box\@ne\if@borderstar\else\kern 2\p@\fi}%
       \vcenter{\if@borderstar\else\kern -\ht\@ne\fi%
         \unvbox\z@\kern-\if@borderstar2\fi\baselineskip}%
         \if@borderstar\kern-2\@tempdima\kern2\p@\else\,\fi\right\@secondoftwo#1 $%
     }\null \;\vbox{\kern\ht\@ne\box\tw@}%
   \endgroup
   }
\makeatother

\newtheorem{theorem}{Theorem}[section] 
\newtheorem{proposition}[theorem]{Proposition}
\newtheorem{lemma}[theorem]{Lemma}

\newtheorem{example}[theorem]{Example}
\newtheorem{definition}[theorem]{Definition}
\newtheorem{remark}[theorem]{Remark} 

\newtheorem{corollary}[theorem]{Corollary}

\newtheorem{notation}[theorem]{Notation}
\newtheorem{weight}[theorem]{Remark (Weight Function)}

\newtheorem{monomial}[theorem]{Weighted Monomial Order}

\newtheorem{coordinatering}[theorem]{Systematic Procedure (off-diag coords of $\mathbb{C}[\mu^{-1}(0)^{rss}]$)}

\allowdisplaybreaks 

\begin{document}

\title[The regular semisimple locus of the affine quotient by the Borel]{The regular semisimple locus of the affine quotient of the cotangent bundle of the Grothendieck-Springer resolution}  
\author{Mee Seong Im} 
\address{Department of Mathematics, University of Illinois at Urbana-Champaign, Urbana, IL 61801 USA}
\email{mim2@illinois.edu}  
\address{Department of Mathematical Sciences, United States Military Academy, West Point, NY 10996 USA}
\email{meeseongim@gmail.com} 
\date{\today}

\begin{abstract}  
Let $G= GL_n(\mathbb{C})$, the general linear group over the complex numbers,  and let $B$ be the set of invertible upper triangular matrices in $G$.  
Let $\mathfrak{b}=\lie(B)$.   
For $\mu:T^*(\mathfrak{b}\times \mathbb{C}^n)\rightarrow \mathfrak{b}^*$, where $\mathfrak{b}^*\cong \mathfrak{g}/\mathfrak{u}$ and $\mathfrak{u}$ being strictly upper triangular matrices in $\mathfrak{g}=\lie(G)$, 
we prove 
that the Hamiltonian reduction $\mu^{-1}(0)^{rss}/\!\!/B$ of the extended regular semisimple locus  $\mathfrak{b}^{rss}$ of the Borel subalgebra is smooth, affine, reduced, and 
scheme-theoretically isomorphic to a dense open locus of $\mathbb{C}^{2n}$. 
We also show that the $B$-invariant functions on the regular semisimple locus of the Hamiltonian reduction of $\mathfrak{b}\times \mathbb{C}^n$ 
arise as the trace of a certain product of matrices. 
\end{abstract}   


\maketitle

\bibliographystyle{amsalpha}  
 
\setcounter{tocdepth}{1}

\section{Introduction}\label{section:introduction}

Grothendieck-Springer resolutions and Springer resolutions, which can be defined for any semisimple Lie algebra or reductive algebraic group\footnote{The Springer resolution for an algebraic group is defined in a similar manner (cf. \cite{Springer69}): letting $U$ to be the set of unipotent elements in $G$ and $\mathcal{B}$ to be the set of all Borel subgroups of $G$, the Springer resolution of $U$ is 
$\widetilde{U}=\{ (u,B)\in U\times \mathcal{B}: u \in B\}$. Since $\mathcal{B}$ can be endowed with a structure of a variety such that $\mathcal{B}\cong G/B$ for any Borel subgroup $B$, $\mathcal{B}$ is irreducible, proper, smooth, and homogeneous under the $G$-action.},  
are fundamental and important objects in representation theory and algebraic geometry, and we give their constructions as follows.  
Throughout this manuscript, 
we will restrict to the setting when $G$ is the general linear group $GL_n(\mathbb{C})$ over the complex numbers. 
Let $G$ act on its Lie algebra $\mathfrak{g}:=\mathfrak{gl}_n=\lie(G)$ by conjugation, and let   
$\psi: \mathfrak{g}\twoheadrightarrow\mathfrak{g}/\!\!/G =\spec(\mathbb{C}[\tr(r),\ldots,\det(r)])=\mathbb{C}^n$ be the adjoint quotient map. 
All elements $r\in \mathfrak{g}$ such that $\psi(r)=0$ 
are those matrices whose characteristic polynomial is of the form $p_r(t)=t^n$. This implies that $r$ is a nilpotent matrix, and writing 
$\mathcal{N}:=\psi^{-1}(0)$, the preimage $\mathcal{N}$ of zero under $\psi$ is called the nilpotent cone;  
it is a normal, reduced, closed subvariety of $\mathfrak{g}$ (cf. \cite{Kostant63}). 
Let $B$ be the set of invertible upper triangular matrices in $G$ and let  
$\mathfrak{b}=\lie(B)$. 
Let $G/B$ be the flag variety parameterizing Borel subalgebras in $\mathfrak{g}$.  
Then writing 
$\widetilde{\mathfrak{g}} := G\times_B \mathfrak{b} =\{ (x,\mathfrak{b})\in \mathfrak{g}\times G/B:x\in \mathfrak{b} \}$
and 
$\widetilde{\mathcal{N}} := G\times_B \mathfrak{n}= \{ (x,\mathfrak{b})\in \mathcal{N}\times G/B: x\in \mathfrak{b} \}$, 
we have a commuting diagram 
\[ \xymatrix{
\widetilde{\mathfrak{g}} \ar@{->>}[rr]^{\pi_1} & & \mathfrak{g} \ar@{->>}[rr]^{\psi} & & \mathfrak{g}/\!\!/G = \mathbb{C}^n  \\  
& & & & \\ 
\widetilde{\mathcal{N}}  \ar@{^{(}->}[uu]  \ar@{->>}[rr]^{p_1}   & & \mathcal{N} \ar@{^{(}->}[uu]  \ar@{->>}[rr]^{\phi} & & \mathcal{N}/\!\!/G  = \{ 0\}, \ar@{^{(}->}[uu]       \\
} 
\] 
where 
$\widetilde{\mathcal{N}}  \stackrel{p_1}{\twoheadrightarrow} \mathcal{N}$ is called the Springer resolution of nilpotent elements in $\mathfrak{g}$ while 
$\widetilde{\mathfrak{g}} \stackrel{\pi_1}{\twoheadrightarrow}\mathfrak{g}$ is called the Grothendieck-Springer resolution for the Lie algebra $\mathfrak{g}$ (cf. \cite{MR0442103}, \cite{Steinberg74} \cite{Steinberg76}), with $\widetilde{\mathcal{N}}\cong T^*(G/B)$  (cf. Chapter 3 in \cite{MR2838836}, \cite{Dol-Gold84}).  
The Springer resolution is a symplectic resolution of the singular symplectic variety $\mathcal{N}$, while, roughly speaking, the commuting diagram above exhibits $\tilde{\mathfrak{g}} \rightarrow \mathfrak{g}$ as a versal Poisson deformation of the Springer resolution.

 An important result known as the Springer correspondence gives a bijection between irreducible representations of the Weyl group (which is the symmetric group $S_n$ of $n$ letters when $G$ is the general linear group) and unipotent conjugacy classes of $G$, which are also parametrized by the partitions of $n$ through the theory of Jordan normal forms; 
 furthermore, for a unipotent conjugacy class $\mathcal{O}$ and a fixed element $u\in \mathcal{O}$, the corresponding irreducible representation of $S_n$ is the cohomology group $H^{2\dim \mathcal{B}_u}(\mathcal{B}_u; \mathbb{Q})$, where $\mathcal{B}_u$   is the set of Borel subgroups of $G$ in the Springer resolution of the algebraic group $U$ containing $u$  
 (cf. \cite{Borho-Mac81}, Chapter 3 in \cite{MR2838836}, \cite{MR1649626}, Chapter 9 of \cite{Humphreys95}, \cite{Lusztig-Spaltenstein-85}, \cite{Spaltenstein82}, \cite{Springer78}).  
In other words, 
the Weyl group acts on cohomology groups of fibers of the Springer resolution.  Moreover,  
Steinberg and Spaltenstein, see \cite{Steinberg74} and \cite{Spaltenstein82}, study the fibers of the Springer resolution and give explicit dimension formulas for the fibers, which are now fundamental algebro-geometric properties of the Springer resolution.

One has an isomorphism $\widetilde{\mathcal{N}} \cong T^*(G/B)$ that identifies the Springer resolution with the moment map $T^*(G/B) \rightarrow \mathcal{N}\subset \mathfrak{g}\cong \mathfrak{g}^*$, where we identify $\mathfrak{g}$ and $\mathfrak{g}^*$ via the trace pairing. 
Dualize the moment map to obtain the comoment map 
$\mathfrak{g} \rightarrow \Gamma(T_{G/B}) \subset \mathbb{C}[T^*(G/B)]$, which coincides with the infinitesimal Lie algebra action $\mathfrak{g}\rightarrow T(G/B)$. 
Then 
by quantizing the infinitesimal action, we obtain a map 
$U(\mathfrak{g})\rightarrow  \mathcal{D}_{G/B}^{\mathcal{L}(\lambda)}$ 
from the universal enveloping algebra to global differential operators on the flag variety twisted by a line bundle $\mathcal{L}(\lambda)$. 
In fact, highest weight representations of $U(\mathfrak{g})$ 
are realized by $B$-equivariant $\mathcal{D}_{G/B}^{\mathcal{L}(\lambda)}$-modules on $G/B$, i.e., $\mathcal{D}^{\mathcal{L}(\lambda)}$-modules on $B\backslash G/B$. 
The map $U(\mathfrak{g}) \rightarrow \mathcal{D}_{G/B}$ and its twisted analogues provide the induction functors appearing in Beilinson-Bernstein localization (a geometric characterization) between $U(\mathfrak{g})$-modules with a given central character and $\mathcal{D}_{G/B}^{\mathcal{L}(\lambda)}$-modules on $G/B$ for generic parameters (cf. \cite{Beilinson83}, \cite{Beilinson-Bernstein-81}, \cite{Hotta-Takeuchi-Tanisaki}).


Now, there is another important object in algebraic geometry called the  
Hilbert scheme, which is a parameter space for closed subschemes of a projective scheme; it is a disjoint union of schemes corresponding to the Hilbert polynomial of the subschemes of the projective scheme.   
In particular, the Hilbert scheme $\Hilb^{n}(\mathbb{C}^2)$ of $n$ points on a complex plane has been extensively studied, for example, see  
\cite{Ginzburg-Nakajima-quivers}, \cite{MR1711344}, \cite{Nakajima16}. 
In \cite{MR1711344}, 
Nakajima gives a description of $\Hilb^{n}(\mathbb{C}^2)$ as the geometric invariant theory (GIT) quotient  
$\mu_{\text{Nak}}^{-1}(0)/\!\!/_{\det}G \cong \mu_{\text{Nak}}^{-1}(0)/\!\!/_{\det^{-1}}G $
of the $G$-equivariant moment map 
$\mu_{\text{Nak}}:T^*(\mathfrak{g}\times \mathbb{C}^n)\rightarrow \mathfrak{g}^*$, given by $(r,s,i,j)\mapsto [r,s]+ij$,  
where 
$(r,s,i,j)\in \mathfrak{g}\times \mathfrak{g}^*\times \mathbb{C}^n\times (\mathbb{C}^n)^* \cong T^*(\mathfrak{g}\times \mathbb{C}^n)$, a quadruple where $r$ and $s$ are $n\times n$ complex matrices, $i$ is a vector, and $j$ is a covector.

Let $B$ act on the vector space $\mathfrak{b}\times \mathbb{C}^n$ via  
$b.(r,i)=(brb^{-1},bi)$. This action is induced onto the cotangent bundle $T^*(\mathfrak{b}\times \mathbb{C}^n) = \mathfrak{b}\times \mathfrak{b}^*\times \mathbb{C}^n\times (\mathbb{C}^n)^*$ of $\mathfrak{b}\times \mathbb{C}^n$ as: 
\[ 
b.(r,s,i,j)= (\Ad_b(r), \Ad_b^*(s),bi, jb^{-1})= (brb^{-1},\overline{bsb^{-1}}, bi, jb^{-1}), 
\]  
 where $\mathfrak{b}^*\cong \mathfrak{g}/\mathfrak{u}$,  
$\mathfrak{u}$ is the strictly upper triangular matrices in $\mathfrak{g}$, 
and 
$\overline{v}:\mathfrak{g}^* \rightarrow \mathfrak{b}^*$ is the canonical projection map 
(the identification between the dual $\mathfrak{b}^*$ of $\mathfrak{b}$ and $\mathfrak{g}/\mathfrak{u}$ is given by the bilinear pairing $\mathfrak{b}\times \mathfrak{g}\twoheadrightarrow \mathbb{C}$, $(r,s)\mapsto \tr(rs)$ which factors through the  bilinear, nondegenerate pairing 
$\mathfrak{b}\times \mathfrak{g}/\mathfrak{u} \rightarrow \mathbb{C}$).  
 The infinitesimal action of $B$ induces the map $a:\mathfrak{b}  \rightarrow 
 \Gamma(T_{\mathfrak{b} \times \mathbb{C}^n}) \subset \mathbb{C}[T^*(\mathfrak{b} \times \mathbb{C}^n)]$ 
 which is given as $a(v)(r,i)=\frac{d}{dt}(g_t.(r,i))|_{t=0}=([v,r],vi)$, where $g_t=\exp(tv)$.
  Dualizing this map gives us 
  the moment map $\mu:T^*\left( \mathfrak{b}\times\mathbb{C}^n \right)\rightarrow \mathfrak{b}^*$,   
  where $(r,s,i,j)$ is mapped to $\ad^*_r(s) + \overline{a^*(ij)}$, where  $a:\mathfrak{g}\rightarrow\End(\mathbb{C}^n)$ is a representation of $\mathfrak{g}$ on $\mathbb{C}^n$. 
  Since $\mathfrak{g}$ is the Lie algebra of the general linear group, the pullback map $a^*:\End(\mathbb{C}^n)^*\rightarrow \mathfrak{g}^*$ is given by 
$a^*(ij)=ij$. 
Thus in our setting, $(r,s,i,j)\mapsto \overline{[r,s]+ij}$. 

In this manuscript, we study the $B$-equivariant moment map $\mu:T^*(\mathfrak{b}\times \mathbb{C}^n)\rightarrow \mathfrak{b}^*$ since there is a close relationship between $\widetilde{\mathfrak{g}}$ and the $B$-moment map $\mu$, which we will now explain. 
Let $G$ act on $G\times \mathfrak{b}\times \mathbb{C}^n$ by 
$g.(g',r,i)=(g'g^{-1},r, gi)$ (note that if $i:G\rightarrow G$ is the inversion map, then its differential $di_{\I}:T_{\I}G\rightarrow T_{\I}G$ is given by $di_{\I}(x)=-x$). 
It induces a 
$G$-moment map $\mu_{G}: T^*(G\times \mathfrak{b}\times \mathbb{C}^n)\rightarrow \mathfrak{g}^*$, where $T^*(G\times \mathfrak{b}\times \mathbb{C}^n) \cong G\times \mathfrak{g}^*\times \mathfrak{b}\times \mathfrak{b}^*\times \mathbb{C}^n\times (\mathbb{C}^n)^*$, 
given by $(g,\theta,r,s,i,j)\mapsto -\theta + a^*(ij)$. 
There is also a $B$-action on $G\times \mathfrak{b}\times \mathbb{C}^n$, which induces 
the moment map $\mu_B:T^*(G\times \mathfrak{b}\times \mathbb{C}^n)\rightarrow \mathfrak{b}^*$ given by 
$(g,\theta,r,s,i,j)\mapsto \overline{\Ad_g^*(\theta)}+\ad_r^*(s)$. 
We define the moment map $\mu_{G\times B}:T^*(G\times \mathfrak{b}\times \mathbb{C}^n)\rightarrow \mathfrak{g}^*\times \mathfrak{b}^*$ 
as 
\begin{align*}
\mu_{G\times B}(g,\theta,r,s,i,j)=(\mu_G(g,\theta,r,s,i,j),\mu_B(g,\theta,r,s,i,j)).
\end{align*} 
So 
$\mu_{G\times B}^{-1}(0)=\{ (g,\theta,r,s,i,j)\in G\times \mathfrak{g}\times \mathfrak{b}\times \mathfrak{b}^*\times \mathbb{C}^n\times (\mathbb{C}^n)^*: \theta = ij, \overline{g\theta g^{-1}}=-\ad_r^*(s) \}$. 
Since $g\in G$ is a free parameter, apply the $G$-action so that $g=1$.  
We are now able to deduce the following (see also Proposition 3.2 and Corollary 3.3 in \cite{Nevins-GSresolutions}):  
\begin{proposition} 
The inclusion $\mu^{-1}(0)\hookrightarrow \mu_{G\times B}^{-1}(0)$ given by 
$(r,s,i,j)\mapsto (1, ij,r,s,i,j)$ induces a bijection between $B$-orbits on $\mu^{-1}(0)$ and $G\times B$-orbits on $\mu_{G\times B}^{-1}(0)$. This gives an isomorphism 
\begin{align*} 
\mu^{-1}(0)/B \cong T^*(\widetilde{\mathfrak{g}}\times \mathbb{C}^n/G)
\end{align*} 
of quotient stacks. 
\end{proposition}

We will study the $B$-moment map rather than directly investigate the cotangent bundle of extended Grothendieck-Springer resolution.    
For each algebraic character $\chi:B\rightarrow\mathbb{C}^*$, 
we would like to understand the Hamiltonian reduction $\mu^{-1}(0)/\!\!/_{\chi}B$ 
of $\mathfrak{b}\times \mathbb{C}^n$ twisted by $\chi$ (for various $\chi$) and relate the scheme to other well-known schemes, such as, the Hilbert scheme of $n$ points on a complex plane, i.e., 
construct a morphism $\mu^{-1}(0)/\!\!/_{\chi}B\rightarrow \mu^{-1}(0)/\!\!/B$ and relate it to 
    $(\mathbb{C}^2)^{[n]}\stackrel{}{\twoheadrightarrow} \mathbb{C}^{2n}/S_n$.  
     
The flag Hilbert scheme on a complex plane is defined to be 
\begin{align*}
\FHilb^n(\mathbb{C}^2) =  
\{ I_n \subseteq \ldots \subseteq I_1 \subseteq I_0 = \mathbb{C}[x,y]: \dim_{\mathbb{C}} \mathbb{C}[x,y]/I_i = i
\}, 
\end{align*}
which also has the following description: let $B^-$ be lower triangular matrices in $G$ and let $\mathfrak{b}^- :=\lie(B^-)$. Let $\mathfrak{u}^-$ be nilpotent matrices in $\mathfrak{b}^-$. Then 
\begin{align*}
\FHilb^n(\mathbb{C}^2)= 
\{ (x,y,i)\in \mathfrak{b}^-\times \mathfrak{b}^-\times \mathbb{C}^n: [x,y]=0, x^a y^b i \mbox{ span }\mathbb{C}^n
\}/B^-.
\end{align*} 
The flag Hilbert schemes are singular for large $n\gg 0$, 
reducible, and their dimensions are much greater than their expected dimensions. 
They are currently of great interest in quantum topology and categorical representation theory because of the correspondence between the Koszul complexes of the torus fixed points on the flag Hilbert scheme and idempotents in the category of Soergel bimodules  
(cf. \cite{GNR16}). It would be interesting to construct explicit maps from $\mu^{-1}(0)/\!\!/_{\chi} B$ to the flag Hilbert scheme for appropriate choices of $\chi$.

The cotangent bundle $T^*(\mathfrak{b}\times \mathbb{C}^n)$ 
can also be viewed as a certain filtered quiver representation space using a double framed Jordan quiver using what is known as universal quiver flags (cf. Section $2$ of \cite{Cerulli-Irelli-Feigin-Reineke}, Section $2$ of \cite{Craw11}, Chapter $3$ in \cite{Im-doctoral-thesis}). 
The quiver flag varieties appear in the geometric interplay of Khovanov-Lauda-Rouquier (also known as KLR or quiver Hecke) algebras (\cite{MR2525917}, \cite{MR2763732}, \cite{Przezdziecki15}, \cite{Rouquier-2-Kac-Moody-algebras}, \cite{Stroppel-Webster-quiver-schur-algebras-q-fock-space}) and in Lusztig's geometric representation of the upper half $U^+$ of the universal enveloping algebra of a Kac-Moody algebra (\cite{MR1035415}, \cite{MR1182165}, \cite{MR1758244}), 
so the importance of the (Grothendieck-)Springer resolution is paramount.

Thus a good understanding of the geometry of the Hamiltonian reduction $\mu^{-1}(0)/\!\!/_{\chi}B$ of $\mathfrak{b}\times \mathbb{C}^n$ by $B$ is useful and valuable.  
On the other hand, standard results from geometric invariant theory do not apply since it is a quotient by a nonreductive group.

Now, before we can construct the GIT or affine quotient, $\mu^{-1}(0)/\!\!/_{\chi}B$ or $\mu^{-1}(0)/\!\!/B$, respectively, we need to show that $\mu^{-1}(0)$ is a complete intersection so that the preimage of zero has appropriate number of irreducible components, i.e., $2^n$ to be exact, in our setting (see \cite{Nevins-GSresolutions} for more detail). 
Although the author has made progress in this direction, it remains an open problem. 
This manuscript, however, gives a complete answer for the open locus of points represented by regular semisimple matrices of a Borel subalgebra, both set-theoretically and scheme-theoretically.

\begin{definition}\label{definition:distinct-r-eigenvalues}
Let $\mu^{-1}(0)^{rss}$ be the set of quadruples 
\[ \{ (r,s,i,j)\in \mu^{-1}(0): r \mbox{ has distinct eigenvalues}\}.  
  \]  
\end{definition}

Let $r=(r_{\iota\gamma})$ be an $n\times n$ matrix in $\mathfrak{b}$. We write $\diag(r)$ to mean the diagonal matrix 
with entries $(r_{11},r_{22},\ldots, r_{nn})$ along its main diagonal.  

Let $\I$ be the $n\times n$ identity matrix.

\begin{proposition}\label{proposition:map-from-B-rss-to-complex-space}
For $(r,s,i,j)$ in $\mu^{-1}(0)^{rss}$, 
choose $b\in B$ as in Proposition~\ref{proposition:distinct-ev-diagonalizable} so that 
$(\Ad_b(r)$, $\Ad_b^*(s)$, $bi$, $jb^{-1})$ $=$ $(\diag(r)$, $s'$, $i'$, $j')$. 
Then each diagonal coordinate function of $s'=(s_{\iota\gamma}')$ is 
\begin{equation}\label{eqn:s-single-prime-first-prop}
s_{\iota\iota}'=
 \left[ \tr\left( \prod_{
1\leq k\leq n, k\not=\iota 
} l_k(r) \right)\right]^{-1} 
  \tr \left(\prod_{
1\leq k\leq n, k\not=\iota
} l_k(r) \:s\right),  
\end{equation} 
where $l_k(r) = r -r_{kk}\I$. 
\end{proposition}

\begin{definition}\label{definition:definition-of-Delta-n-set}
We denote $\Delta_n\subseteq \mathbb{C}^{2n}$ to be the set $\{ (x_{11},\ldots, x_{nn},0,\ldots, 0): x_{\iota\iota}=x_{\gamma\gamma} 
$\mbox{ for some $\iota\not=\gamma$}$ \}$.
\end{definition}

Thus $\mathbb{C}^{2n}\setminus \Delta_n$ is the locus  
$\{ (x_{11},\ldots, x_{nn},y_{11},\ldots,y_{nn}): x_{\iota\iota}\not= x_{\gamma\gamma} \mbox{ whenever } \iota\not= \gamma
\}.
$ 
We now state the main theorems proved in this paper. 

\begin{theorem}\label{theorem:rss-extending-to-all-of-the-reg-semi-stable-locus}
The map 
$P:\mu^{-1}(0)^{rss}\rightarrow\hspace{-8pt}\rightarrow \mathbb{C}^{2n}\setminus \Delta_n$  
given by sending 
\[(r,s,i,j)\mapsto (r_{11},\ldots, r_{nn},s_{11}',\ldots, s_{nn}'),
\] where 
$s_{\iota\iota}'$ in \eqref{eqn:s-single-prime-first-prop}  
is a regular, well-defined, and $B$-invariant surjective map separating orbit closures. 
\end{theorem}

\begin{theorem}\label{theorem:rss-locus-to-complex-space-B-invariant}
The map $P$ in Theorem~\ref{theorem:rss-extending-to-all-of-the-reg-semi-stable-locus} descends 
to a set-theoretic bijective homeomorphism 
$p: \mu^{-1}(0)^{rss}/\!\!/B\rightarrow \mathbb{C}^{2n}\setminus \Delta_n$, where 
\[ \overline{B.(r,s,i,j)} \mapsto (r_{11},\ldots, r_{nn},s_{11}',\ldots, s_{nn}') 
\]  
with $s_{\iota\iota}'$ is given in \eqref{eqn:s-single-prime-first-prop} and $\overline{B.(r,s,i,j)}$ is the $B$-orbit closure of $(r,s,i,j)$. In fact, $p$ induces an isomorphism of varieties.  
\end{theorem}

\subsection{Summary of the sections}\label{subsection:summary}
Section~\ref{subsection:orthogonal-projection-operators} investigates properties of a certain set of orthogonal idempotents and their action on a regular semisimple upper-triangular matrix. An orthogonal idempotent $L^{\iota}$ is defined in Definition~\ref{defn:coordinatefunction-of-product-of-rs}, where $1\leq \iota\leq n$. Intuitively, given a regular semisimple matrix $r\in \mathfrak{b}^{rss}$, 
$r$ is diagonalizable with pairwise distinct eigenvalues $r_{kk}$ and eigenvector $v_k = e_k + \sum_{j<k} a_j e_j$. The matrix $l_k(r)=r-r_{kk}\I$ of rank $n-1$ has the same eigenvectors as $r$, and $v_k$ is in the kernel of $l_k(r)$. It follows that multiplying $l_k(r)$, $1\leq k\leq n$, $k\not=\iota$, has all $v_j$, $j\not=\iota$, in the kernel and the vector $v_{\iota}$ as an eigenvector with nonzero eigenvalue. 
This section discusses patterns among the idempotents $L^{\iota}$ and their interactions with $r\in \mathfrak{b}^{rss}$. We also give some results of the interactions between $L^{\iota}$ and $s\in \mathfrak{b}^*$. 
The results in this section are not obvious to some readers so we give a full detail on the idempotents. 
In Section~\ref{subsection:Borel-fixed-points}, we show that certain points remain invariant under the Borel action.

In Section~\ref{section:proof-of-Brss-complex-space}, we prove Proposition~\ref{proposition:map-from-B-rss-to-complex-space}, showing that when $r\in \mathfrak{b}^{rss}$ is diagonalized by the Borel element in  \eqref{eqn:diagonalizingmatrix-B}, the diagonal coordinates of $s'$ are the expected rational functions.  
We then study $B$-orbits on the regular semisimple locus in Section~\ref{section:B-orbits-rss-locus}. In particular, we show that closed orbits in $\mu^{-1}(0)^{rss}$ must be of the form $(r,s,0,0)$ in Remark~\ref{remark:exhausting-all-closed-orbits} and Proposition~\ref{proposition:openlocus-affine-quotient}, 
where $r$ is a regular semisimple element in $\mathfrak{b}$. 

In Section~\ref{section:proof-of-reg-semistable-locus}, we give a proof of Theorem~\ref{theorem:rss-extending-to-all-of-the-reg-semi-stable-locus}. 
  Section~\ref{section:coordinate-ring-rss-locus} thoroughly investigates the coordinate ring of the regular semisimple locus by changing coordinates (Section~\ref{subsection:changing-coordinates}), explicitly describing the 
  $B$-invariant subalgebra of $\mathbb{C}[\mu^{-1}(0)^{rss}]$ (Section~\ref{subsection:B-invariantfunctions}),  and introducing an initial ideal with respect to a weighted monomial ordering in order to show that the set of functions $F_{\iota}(r,s,i,j)=\tr(jL^{\iota}i)$ forms a regular sequence in the ring $\mathbb{C}[T^*(\mathfrak{b}\times \mathbb{C}^n)^{rss}]$ (Section~\ref{subsection:initial-ideal-regular-sequence}).  
 In Section~\ref{section:proof-of-last-proposition-coordinate-ring}, we prove Theorem~\ref{theorem:rss-locus-to-complex-space-B-invariant}, showing a homeomorphism between $B$-affine quotient of the locus $\mu^{-1}(0)^{rss}$ and an open dense subset of $\mathbb{C}^{2n}$.

\subsection{Acknowledgement}\label{subsection:acknowledgement}
The author would like to thank Thomas Nevins for helpful conversations, and the referee for immensely useful remarks on this manuscript. The author was supported by NSA grant H98230-12-1-0216, by Campus Research Board, and by NSF
grant DMS 08-38434.

\section{Preliminaries}\label{section:preliminaries}  
  
We assume throughout this paper that $r\in \mathfrak{b}^{rss}$, an 
 $n\times n$
upper triangular matrix with pairwise distinct eigenvalues.  

\subsection{Properties of orthogonal idempotents in \texorpdfstring{$B$}{B}}\label{subsection:orthogonal-projection-operators} 

We begin with a discussion of a certain set of orthogonal idempotents in $B$. 

\begin{notation}\label{notation:empty-sum-empty-product-convention}
We will use the convention that an empty sum is defined to be 0 while an empty product is defined to be 1; that is, if $\gamma <\iota$, 
\[ \sum_{k= \iota}^{\gamma} f(k) := 0 
\hspace{4mm}
\mbox{ and } 
\hspace{4mm}
\prod_{k=\iota}^{\gamma} f(k) := 1.  
\]  
We also note that 
\[ \sum_{\iota < k_1 < \ldots < k_v < \mu} f(k_i) :=0 
\hspace{4mm}
\mbox{ if }
  v\geq \mu-\iota. 
\]      
\end{notation}

\begin{definition}\label{defn:coordinatefunction-of-product-of-rs}
For $l_k(r)=r-r_{kk}\I$, 
we define 
\begin{equation}\label{eqn:n-idempotents-L-iota}
L^{\iota} := \left[\tr\left(\prod_{1\leq k\leq n, k\not= \iota} l_k(r)\right)\right]^{-1} \prod_{1\leq k\leq n, k\not= \iota} 
l_k(r) 
\end{equation}
and 
will write  
\[ 
L_{\gamma\mu}^{\iota} = 
\left( \left[\tr\left(\prod_{1\leq k\leq n, k\not= \iota} l_k(r)\right)\right]^{-1} \prod_{1\leq k\leq n, k\not= \iota} 
l_k(r) \right)_{\gamma\mu} 
\]   to denote the coordinate functions of $L^{\iota}$. 
\end{definition} 
In Lemma~\ref{lemma:Baction-on-Liota-operator} and 
Propositions~\ref{proposition:diagonal-rss-natural-matrix-representation} and \ref{proposition:B-invariant-functions-involving-s}, we will write $L^{\iota}(r)$ instead of $L^{\iota}$ since $L^{\iota}$ depends on $r$.

\begin{lemma}\label{lemma:lkrmatrices-orthogonal-rss}  
The $L^{\iota}$'s form mutually orthogonal idempotents. That is, 
we have 
\[ \tr(L^{\iota})=1,\:\: \left( L^{\iota} \right)^2 = L^{\iota}\:\: \mbox{ and }
\:\: L^{\iota} L^{\gamma}=0 
\] for any $\iota\not=\gamma$. 
In particular, any row of $L^{\iota}$
is orthogonal to any column of $L^{\gamma}$ 
for $\iota\not=\gamma$. 
\end{lemma}

\begin{proof}  
It is clear that $\tr(L^{\iota})=1$ 
since $L_{\mu\alpha}^{\iota}=0$ if $\mu>\iota$ or $\alpha <\iota$, and the only nonzero diagonal entry is  $L_{\iota\iota}^{\iota}=1$. 

Now suppose 
$\iota=\gamma$.  
Then 
\begin{gather}\label{gather:L-iota-operator-options}
\begin{aligned}
L_{\mu\alpha}^{\iota} =0 &\mbox{ if } \mu>\iota \mbox{ or } \alpha <\iota, \mbox{ and }  \\
L_{\alpha\nu}^{\iota} =0 &\mbox{ if } \alpha >\iota \mbox{ or } \nu<\iota.   \\
\end{aligned}
\end{gather}
So 
\[\begin{aligned}
\left(L^{\iota}L^{\iota} \right)_{\mu\nu} 
    &= 
\sum_{\alpha=1}^n L_{\mu\alpha}^{\iota} L_{\alpha\nu}^{\iota} \\
  &= 
  \sum_{\alpha=\iota}^{\iota} L_{\mu\alpha}^{\iota} L_{\alpha\nu}^{\iota} \\
   &= 
L_{\mu\iota}^{\iota} L_{\iota\nu}^{\iota}  \\
&= \left\{ 
\begin{aligned} 
L_{\mu\iota}^{\iota} = L_{\mu\nu}^{\iota} 
		\quad
			&\mbox{ if } \mu < \iota, \nu = \iota, \\
1\;\; = L_{\mu\nu}^{\iota} 
		\quad  
			&\mbox{ if } \mu=\iota, \nu = \iota, \\
L_{\iota\nu}^{\iota} = L_{\mu\nu}^{\iota}  
		\quad 
			&\mbox{ if } \mu = \iota, \nu > \iota, \\  
 L_{\mu\nu}^{\iota} 
		\quad 
 			&\mbox{ if } \mu <\iota, \nu > \iota, \\  
0\;\; = L_{\mu\nu}^{\iota}   
		\quad  
 			&\mbox{ if } \mu > \iota \mbox{ or } \nu<\iota, \\  
\end{aligned}  
\right. \\ 
\end{aligned}  
\]
where the second equality holds by \eqref{gather:L-iota-operator-options}.  
It is a direct calculation that 
$L_{\mu\iota}^{\iota} L_{\iota\nu}^{\iota} = L_{\mu\nu}^{\iota}$.  
Thus $\left( L^{\iota}\right)^2 = L^{\iota}$.

For  
$\iota > \gamma$, 
\[
\left( L^{\iota}L^{\gamma}\right)_{\mu\nu} 
=  \sum_{\alpha=1}^n L_{\mu\alpha}^{\iota}L_{\alpha\nu}^{\gamma}=0
\] since $L_{\mu\alpha}^{\iota}=0$ for each  
$\alpha <\iota$ 
and $L_{\alpha\nu}^{\gamma}=0$ for each $\alpha>\gamma$. 
 Thus $L^{\iota}L^{\gamma}=0$ whenever $\iota >\gamma$.

Finally for 
$\iota < \gamma$,  
$L^{\iota}L^{\gamma} = L^{\gamma} L^{\iota} =0$  since $L^{\iota}$'s commute and by previous case.  
\end{proof} 

\begin{remark}
The $n$ idempotents $L^{\iota}$ 
in \eqref{eqn:n-idempotents-L-iota}  
sum to the identity: 
$\displaystyle{\sum_{\iota=1}^{n} L^{\iota}}=\I$. 
So the identity element decomposes as the sum of these orthogonal idempotents. 
\end{remark}

\begin{corollary}\label{cor:product-of-r-minus-rkk} 
For $1\leq \iota\leq n$,   
$\displaystyle{
L^{\iota}  r = r L^{\iota}}
$ 
are zero strictly to the left of $\iota$-th column or 
strictly below $\iota$-th row. 
\end{corollary} 

\begin{proof} 
This is deduced from vanishing properties of $L^{\iota}$. 
\end{proof}



%
%
 
\begin{lemma}\label{lemma:lkr-times-r-identity}
For each $\gamma$ and $r\in\mathfrak{b}$,  
$(L^{\iota} r)_{\gamma\gamma} = (rL^{\iota})_{\gamma\gamma} = r_{\iota\iota}L_{\gamma\gamma}^{\iota}$.  
\end{lemma}

\begin{proof}
For a fixed $\iota$, it is clear that $L^{\iota}r=rL^{\iota}$ since $r$ and $L^{\iota}$ commute. We will thus show for each $\gamma$, 
		$(L^{\iota}r)_{\gamma\gamma} = r_{\iota\iota}L_{\gamma\gamma}^{\iota}$.  
So 
\[
  \begin{aligned} 
\left( 
  L^{\iota}r\right)_{\gamma\gamma} &= \sum_{\mu=1}^n 
L_{\gamma\mu}^{\iota} r_{\mu\gamma}  \\
&= \sum_{\mu=\iota}^n L_{\gamma\mu}^{\iota}r_{\mu\gamma} 
\hspace{4mm}
\mbox{ since }L_{\gamma\mu}^{\iota} =0 \;\;\;\mbox{ for }\mu<\iota \\  
&= \left\{  
\begin{aligned}  
\sum_{\mu=\iota}^n 0\cdot r_{\mu\gamma} = 0=r_{\iota\iota}L_{\gamma\gamma}^{\iota} 
			\qquad\qquad\qquad 
			&\mbox{ if }\gamma >\iota, \\
 \sum_{\mu=\iota}^n L_{\gamma\mu}^{\iota} r_{\mu\gamma} 
 	\stackrel{\dagger}{=}		\sum_{\mu=\iota}^{\iota} L_{\iota\mu}^{\iota} r_{\mu\iota} 
	= 	r_{\iota\iota}\cdot 1= r_{\iota\iota}L_{\gamma\gamma}^{\iota}      
			\;\; 
			&\mbox{ if }\gamma = \iota, \\
\sum_{\mu=\iota}^n L_{\gamma\mu}^{\iota} r_{\mu\gamma} 
	\stackrel{\ddagger}{=}   \sum_{\mu=\iota}^n L_{\gamma\mu}^{\iota}\cdot 0 = r_{\iota\iota}L_{\gamma\gamma}^{\iota}  
			\qquad 
			&\mbox{ if }\gamma < \iota,  \\  
\end{aligned} 
\right.  
\end{aligned}  
\] 
where 
	$\dagger$ holds since $r_{\mu\gamma}=0$ for all $\mu > \gamma$ and since $\gamma=\iota$, 
 		and 
 			$\ddagger$ holds since $r_{\mu\gamma}=0$ for all $\mu > \gamma$ 
 				and since $\mu$ ranges over $\iota \leq \mu\leq n$.  
\end{proof}

%
%
%

\begin{corollary}\label{cor:productoflkr-r}  
Let $\iota=1,2,\ldots, n$,   and $s\in\mathfrak{b}^*$. 
For each $\gamma > \iota$, $\left(L^{\iota}s\right)_{\gamma\mu}=0$.  
\end{corollary}

The matrix  
$\displaystyle{
   L^{\iota} s 
}$ 
in Corollary~\ref{cor:productoflkr-r} 
is zero strictly below $\iota$-th row. 
However,  
$\displaystyle{L^{\iota} s }$ is not zero strictly to the left of $\iota$-th column. 

\begin{proof}  
For $\gamma > \iota$, 
\[ 
\left( \prod_{1\leq k \leq n, k\not=\iota} l_k(r) \: s \right)_{\gamma\mu} 
	= \sum_{\kappa=1}^n \left(\prod_{1\leq k\leq n, k\not=\iota }l_k(r) \right)_{\gamma\kappa} \!\! s_{\kappa\mu}  
			=   
	\sum_{\kappa=1}^n \: 0\cdot s_{\kappa\mu}  = 0,   \\  
\] 
where the second equality holds by Corollary~\ref{cor:product-of-r-minus-rkk}. 
\end{proof}

%

We will see an application of the following Lemma in Section~\ref{subsection:B-invariantfunctions}. 

\begin{lemma}\label{lemma:Baction-on-Liota-operator}
For any $d\in B$, $L^{\iota}(\Ad_d(r))=\Ad_d (L^{\iota}(r))$, where the $B$-action on the operator is by conjugation. 
\end{lemma}

\begin{proof}
For any $d\in B$, 
\[\begin{aligned} 
L^{\iota}(\Ad_d(r)) &=  \left[
\tr\left( \prod_{
1\leq k\leq n, k\not=\iota 
} 
      drd^{-1} -r_{kk}\I   
\right)\right]^{-1}   \left(\prod_{
1\leq k\leq n, k\not=\iota
}    
   drd^{-1} - r_{kk}\I
\right) \\  
&=  \left[\tr\left(  d\left( \prod_{1\leq k \leq n,k\not=\iota} 
     r-r_{kk}\I \right) 
 d^{-1} \right) \right]^{-1}  d   
   \left( 
   \prod_{1\leq k\leq n, k\not=\iota} r-r_{kk}\I 
  \right)
 d^{-1}  
\\ 
&= d \left( \left[\tr\left(   \prod_{1\leq k \leq n,k\not=\iota} 
     r-r_{kk}\I   
  \right) \right]^{-1}    
   \prod_{1\leq k\leq n, k\not=\iota} 
   \left(  r-r_{kk}\I 
 					 \right) 
  \right) d^{-1}   \\    
   &= \Ad_d(L^{\iota}(r)).    \\  
\end{aligned} 
\] 
\end{proof}

%
%
%

%
%
%
%

\subsection{Borel fixed points}\label{subsection:Borel-fixed-points} 
In this section, we prove some basic facts about the action of $B$. 
Proposition~\ref{proposition:diagonal-being-invariant} shows that the diagonal entries of an upper triangular matrix do not change under the $B$-conjugation action whilst Lemma~\ref{lemma:s-diagonal-conjugation-still-diagonal} shows that diagonal entries of a diagonal matrix are preserved under the $B$-action.

\begin{proposition}\label{proposition:diagonal-being-invariant}
For any $r$ in $\mathfrak{b}$ and for any $b$ in $B$, we have $\diag(\Ad_b(r))=\diag(r)$. 
\end{proposition}

\begin{proof}
Denote the entries of $b$, $r$, and $b^{-1}$ as 
$b_{\iota\gamma}$, 
$r_{\iota\gamma}$, and  
$b_{\iota\gamma}'$, respectively,  
with 
$b_{\iota\iota}'=b_{\iota\iota}^{-1}$. 
Then   
$(br)_{\iota\gamma} 
= \displaystyle{\sum_{k=1}^n b_{\iota k}r_{k\gamma}}$. 
Since $b_{\iota k}$ and $r_{\iota k}$ equal 0 whenever $\iota >k$, 
\[ (br)_{\iota\gamma} 
= \left\{ 
 \begin{aligned}
 \sum_{k=\iota}^{\gamma} b_{\iota k}r_{k\gamma}\qquad & \mbox{ if } \iota <\gamma, \\ 
   b_{\iota\iota}r_{\iota\iota}\qquad       \quad    & \mbox{ if } \iota =\gamma, \\ 
   0 \qquad\qquad                               &  \mbox{ if } \iota>\gamma. \\ 
 \end{aligned}
\right. 
\] 
Renaming 
$(br)_{\iota k}$ as $d_{\iota k}$, we obtain  
 $(brb^{-1})_{\iota \gamma}
=
\displaystyle{\sum_{k=1}^n d_{\iota k}b_{k\gamma}'}$. 
Since $d_{\iota k}$ and $b_{\iota k}'$ equal 0 whenever $\iota >k$, 
\[ (brb^{-1})_{\iota\gamma} = \left\{ 
\begin{aligned} 
\sum_{k=\iota }^{\gamma} d_{\iota k}b_{k\gamma}' \qquad 
       & \mbox{ if $\iota <\gamma$}, \\ 
  d_{\iota\iota}b_{\iota\iota }'  \qquad\quad 
       & \mbox{ if $\iota =\gamma$}, \\ 
  0\qquad\qquad 
       &  \mbox{ if $\iota>\gamma$}.  \\   
\end{aligned}
\right. 
\] 
Since $(brb^{-1})_{\iota\iota}=d_{\iota\iota}b_{\iota\iota}'=b_{\iota\iota}r_{\iota\iota}b_{\iota\iota}^{-1} = r_{\iota\iota}$, we are done.  
\end{proof}

\begin{lemma}\label{lemma:s-diagonal-conjugation-still-diagonal}
Let $s=\diag(s)$. For any $b$ in $B$, $\diag(\Ad_b^*(s))=\diag(s)$.  
\end{lemma} 

\begin{proof} 
Restrict $r$ in Proposition~\ref{proposition:diagonal-being-invariant} to those in  $\mathfrak{b}\cap\mathfrak{b}^*=\mathfrak{h}$. 
\end{proof}

For a diagonal matrix $r$ in $\mathfrak{b}$ and for a general $b$ in $B$, 
$brb^{-1}$ need not equal $r$ (that is, $brb^{-1}$ need not be diagonal). 
However, for a diagonal matrix $s$ in $\mathfrak{b}^*$ and for any $b$ in $B$, $\Ad_b^*(s) = s$ (i.e., $\Ad_b^*(\diag(s)) = \diag(s)$) always holds in $\mathfrak{b}^* = \mathfrak{b}/\mathfrak{u}$ since the strictly upper triangular part is killed in $\mathfrak{b}^*$.

\begin{proposition}\label{proposition:diag-entries-r-is-trace}
For each $\iota$, 
$r_{\iota\iota} = \tr(L^{\iota}r)$.   
\end{proposition}  

\begin{proof}
We have 
$\tr(L^{\iota}r)=\tr(r_{\iota\iota}L^{\iota})=r_{\iota\iota}$ where the first equality holds by Lemma~\ref{lemma:lkr-times-r-identity} and the second equality holds by direct calculation ($L_{\gamma\gamma}^{\iota}$ is 0 if $\gamma\not=\iota$  and equals 1 if $\gamma = \iota$).  
\end{proof}

\begin{proposition}\label{proposition:difference-of-r-is-trace}
For each $1\leq \gamma < \nu\leq n$, 
$(r_{\nu\nu} - r_{\gamma\gamma})^{-1} = [\tr((L^{\nu}-L^{\gamma})r)]^{-1}$.  
\end{proposition} 

\begin{proof}
Applying Proposition~\ref{proposition:diag-entries-r-is-trace}, we have  $\tr(L^{\nu}r-L^{\gamma}r)$ $=$ $\tr(L^{\nu}r)-\tr(L^{\gamma}r)$ $=$ $r_{\nu\nu}-r_{\gamma\gamma}$. 
\end{proof}


\begin{proposition}\label{proposition:distinct-ev-diagonalizable}
Let $r$ in $\mathfrak{b}^{rss}$ and let $E_{\iota\iota}$ have 1 in $(\iota,\iota)$-entry and $0$ elsewhere. 
There exists 
\begin{equation}\label{eqn:diagonalizingmatrix-B}
b = \sum_{\iota =1}^n  
 E_{\iota\iota} \: \diag(b) L^{\iota} \in B 
\end{equation}  
which diagonalizes $r$. 
Furthermore, the inverse of this Borel element has matrix representation 
\begin{equation}\label{eqn:diagonalizingmatrix-inverse-of-B}
b^{-1} = \sum_{\iota =1}^n  L^{\iota} 
\diag(b)^{-1} \: E_{\iota\iota}. 
\end{equation}
\end{proposition}

One may replace $\diag(b)$ with $b$ on the right-hand side of \eqref{eqn:diagonalizingmatrix-B} 
and still obtain a Borel matrix that diagonalizes a regular semisimple element $r$. 

\begin{proof}
For $r=(r_{ij})$, 
we will show $brb^{-1}=\diag(r)$, or equivalently, $br=\diag(r)b$. So 
\[ 
\begin{aligned} 
br &=  \sum_{\iota =1}^n  
 \left[\tr\left( \prod_{
1\leq k\leq n, k\not=\iota 
} l_k(r) \right)\right]^{-1} E_{\iota\iota} \: \diag(b) \left(\prod_{
1\leq k\leq n, k\not=\iota
} l_k(r) \right) r\\
&=\sum_{\iota =1}^n  
 \left[ \tr\left( \prod_{
1\leq k\leq n, k\not=\iota 
} l_k(r) \right)\right]^{-1} b_{\iota\iota} E_{\iota\iota} \: r \left(\prod_{
1\leq k\leq n, k\not=\iota
} l_k(r) \right) \\ 
&=\sum_{\iota =1}^n  
 \left[ \tr\left( \prod_{
1\leq k\leq n, k\not=\iota 
} l_k(r) \right)\right]^{-1} \left[ 
\begin{array}{c} 
0 \\ 
\vdots \\ 
b_{\iota\iota} \vec{r}_{\iota} \\ 
\vdots \\ 
0 \\  
\end{array}
\right] \left(\prod_{
1\leq k\leq n, k\not=\iota
} l_k(r) \right) \mbox{ where }\vec{r}_{\iota}\mbox{ is the }
\iota\mbox{-th} \mbox{ row of } r \\ 
&= \sum_{\iota =1}^n  
 \left[ \tr\left( \prod_{
1\leq k\leq n, k\not=\iota 
} l_k(r) \right)\right]^{-1} b_{\iota\iota} r_{\iota\iota} E_{\iota\iota}  \left(\prod_{
1\leq k\leq n, k\not=\iota
} l_k(r) \right)\mbox{ by Corollary~\ref{cor:product-of-r-minus-rkk}} \\ 
&= \sum_{\iota =1}^n  
\left[ \tr\left( \prod_{
1\leq k\leq n, k\not=\iota 
} l_k(r) \right)\right]^{-1} \diag(r) E_{\iota\iota} \diag(b) \left(\prod_{
1\leq k\leq n, k\not=\iota
} l_k(r) \right) \\ 
&= \diag(r)b. \\
\end{aligned} 
\]     

Now, we will prove that $b b^{-1}=\I$. 
The $\iota$-th row of $b$ is $e_{\iota}^* \diag(b)L^{\iota}$ and the 
$\gamma$-th column of $b^{-1}$ is 
$L^{\gamma}\diag(b)^{-1}e_{\gamma}$,  
  where 
   $e_{\gamma}$ is an elementary column vector 
and $e_{\iota}^*$ is an elementary covector. 
By Lemma~\ref{lemma:lkrmatrices-orthogonal-rss},  
\[
e_{\iota}^* \diag(b)L^{\iota}  L^{\gamma}\diag(b)^{-1}e_{\gamma} =
\left\{  
\begin{aligned} 
1 \quad  &\mbox{ if }\iota=\gamma \\  
0 \quad  &\mbox{ if }\iota\not=\gamma.  \\ 
\end{aligned} 
\right. 
\] 

\end{proof}

\begin{remark}
The Borel matrix \eqref{eqn:diagonalizingmatrix-B} has coordinates 
\[ b_{\iota\gamma} = \left\{ 
\begin{aligned} 
0 \;\;\;\;\;\;\;\;\;\;\;\;\;\;\;\;\;\;\;\;\;\;\;\;\;\;\;\;\;\;\; \;\;\;\;\;\;\;\;\;\;\;\;\;\;\;\;\;\;\;\;\;\;\;\;\;\;\;\;\;\;\;\;\;\;\; &\mbox{ if } \iota >\gamma,  \\
b_{\iota\iota}\;\;\;\;\;\;\;\;\;\;\;\;\;\;\;\;\;\;\;\;\;\;\;\;\;\;\;\;\;\; \;\;\;\;\;\;\;\;\;\;\;\;\;\;\;\;\;\;\;\;\;\;\;\;\;\;\;\;\;\;\;\;\;\;\; &\mbox{ if } \iota = \gamma, \\ 
b_{\iota\iota}\left(
\dfrac{r_{\iota\gamma}}{r_{\iota\iota}-r_{\gamma\gamma}}  
+
\sum_{v=1}^{\gamma-\iota-1}\sum_{\iota<k_1<\ldots <k_{v}<\gamma} 
\dfrac{r_{\iota k_1}r_{k_v \gamma} }{(r_{\iota\iota}-r_{k_1k_1})(r_{\iota\iota}-r_{\gamma\gamma})} 
\prod_{u=1}^{v-1} \dfrac{r_{k_uk_{u+1}}}{r_{\iota\iota}-r_{k_{u+1}k_{u+1}}} 
\right) &\mbox{ if } \iota < \gamma, \\ 
\end{aligned}
\right. 
\] 
and the inverse \eqref{eqn:diagonalizingmatrix-inverse-of-B} of $b$ has coordinates 
\[ (b^{-1})_{\iota\gamma} = \left\{ 
\begin{aligned} 
0 \;\;\;\;\;\;\;\;\;\;\;\;\;\;\;\;\;\;\;\;\;\;\;\;\;\;\;\;\;\;\; \;\;\;\;\;\;\;\;\;\;\;\;\;\;\;\;\;\;\;\;\;\;\;\;\;\;\;\;\;\;\;\;\;\;\; &\mbox{ if } \iota >\gamma, \\
b_{\gamma\gamma}^{-1}\;\;\;\;\;\;\;\;\;\;\;\;\;\;\;\;\;\;\;\;\;\;\;\;\;\;\;\;\;\; \;\;\;\;\;\;\;\;\;\;\;\;\;\;\;\;\;\;\;\;\;\;\;\;\;\;\;\;\;\;\;\;\;\;\; &\mbox{ if } \iota = \gamma, \\ 
b_{\gamma\gamma}^{-1}\left(\dfrac{r_{\iota\gamma}}{
r_{\gamma\gamma}-r_{\iota\iota}} 
+
 \sum_{v=1}^{\gamma-\iota-1}
 \sum_{\iota <k_1<\ldots <k_{v}<\gamma} 
\dfrac{r_{\iota k_1}r_{k_v \gamma} }{
(r_{\gamma\gamma}-r_{k_vk_v})
(r_{\gamma\gamma}-r_{\iota\iota})} 
\prod_{u=1}^{v-1} \dfrac{r_{k_uk_{u+1}}}{r_{\gamma\gamma}-r_{k_{u}k_{u}}} 
\right) &\mbox{ if } \iota < \gamma. \\ 
\end{aligned}
\right.  
\]  
\end{remark}

\section{Proof of Proposition \texorpdfstring{$\ref{proposition:map-from-B-rss-to-complex-space}$}{coordinate function of s}}\label{section:proof-of-Brss-complex-space}

We will now prove Proposition~\ref{proposition:map-from-B-rss-to-complex-space}. 

\begin{proposition}\label{proposition:deriving-diag-coordinates-of-bsb-inverse}
For each $\iota$, 
\[ (\Ad_b^*(s))_{\iota\iota} = \tr(L^{\iota}s), 
\] where 
$b$ is given in \eqref{eqn:diagonalizingmatrix-B}.  
\end{proposition}

\begin{proof}  
Fix $\iota$ and let $\gamma \leq \iota$.   
Since  
\[   
\begin{aligned} 
(bs)_{\iota\gamma} &= \sum_{\mu=\iota}^n b_{\iota\mu}s_{\mu\gamma} \\ 
&= b_{\iota\iota}s_{\iota\gamma}+ \sum_{\mu=\iota+1}^n b_{\iota\iota}
\left( \dfrac{ r_{\iota\mu} }{r_{\iota\iota}-r_{\mu\mu} } 
+ \sum_{v=1}^{\mu -\iota-1} \sum_{\iota < k_1 < \ldots < k_v < \mu} 
\dfrac{r_{\iota k_1}r_{k_v \mu} }{(r_{\iota\iota}-r_{k_1 k_1})(r_{\iota\iota}-r_{\mu\mu}) } \prod_{u=1}^{v-1} 
\dfrac{r_{k_u k_{u+1}}}{r_{\iota\iota}-r_{k_{u+1}k_{u+1}} } 
\right) s_{\mu\gamma},  \\ 
\end{aligned} 
\]   
we have  
$\displaystyle{
(\overline{bsb^{-1}})_{\iota\iota} 
= \sum_{\gamma=1}^{\iota} (bs)_{\iota\gamma} (b^{-1})_{\gamma\iota}
}$ 
since $(bs)_{\iota\gamma} (b^{-1})_{\gamma\iota}=0$ for each $\gamma >\iota$.  

Now considering the product of matrices 
$L^{\iota}s$, 
$(L^{\iota}s)_{\gamma\gamma}=0$
for all $\gamma > \iota$ by Corollary~\ref{cor:productoflkr-r}. 
Since  
\[ (bs)_{\iota\gamma} (b^{-1})_{\gamma\iota} =
 \left( 
 L^{\iota}s 
\right)_{\gamma\gamma}  
 \] for each $\gamma$, we conclude 
 \[ (\overline{bsb^{-1}})_{\iota\iota} = \sum_{\gamma=1}^{\iota} 
 \underbrace{(bs)_{\iota\gamma} (b^{-1})_{\gamma\iota}}_{
 (L^{\iota}s)_{\gamma\gamma}  
 } =   \tr(L^{\iota}s) = s_{\iota\iota}'.  
\]   
\end{proof}

\begin{remark}\label{remark:deriving-diag-coordinates-of-bsb-inverse-opposite-order} 
Suppose $\overline{bsb^{-1}}$ is multiplied from right to left than from left to right. 
That is,  
for $\gamma\geq \iota$, 
\[ (sb^{-1})_{\gamma\iota} = \sum_{\mu=1}^{\iota} s_{\gamma\mu}(b^{-1})_{\mu\iota}
\] with 
\[ (\overline{bsb^{-1}})_{\iota\iota} = \sum_{\gamma=\iota}^{n} b_{\iota\gamma}
(sb^{-1})_{\gamma\iota}, 
\] 
   then we have 
\[ b_{\iota\gamma}
(sb^{-1})_{\gamma\iota} 
= 
\left( s L^{\iota} 
\right)_{\gamma\gamma}, 
\] 
and similar as before, for $\gamma < \iota$, 
\[ b_{\iota\gamma}
(sb^{-1})_{\gamma\iota} =0 
= 
\left( sL^{\iota}
 \right)_{\gamma\gamma}.  
\] 
\end{remark}

\section{\texorpdfstring{$B$}{B}-orbits on the \texorpdfstring{$rss$}{rss}-locus}\label{section:B-orbits-rss-locus}

We now analyze the affine quotient of $\mu^{-1}(0)^{rss}$ by $B$. 

\begin{proposition}\label{proposition:s-in-HR-is-diagonalizable}
If $(r,s,i,j)$ is in $\mu^{-1}(0)^{rss}$ with $i$ or $j$ equaling zero, then $s$ is diagonal.  
\end{proposition} 

We prove Proposition~\ref{proposition:s-in-HR-is-diagonalizable} using strong induction. 

\begin{proof} 
Let $r=(r_{\iota\gamma})$ and $s=(s_{\iota\gamma })$. 
Let $n=2$. Then 
\begin{align*} 
[r,s] &= \left( 
\begin{array}{cc} 
r_{11} & r_{12} \\ 
 0 & r_{22} \\ 
\end{array}
\right) 
\left( 
\begin{array}{cc} 
s_{11} & 0 \\ 
 s_{21} & s_{22} \\ 
\end{array}
\right) 
-
\left( 
\begin{array}{cc} 
s_{11} & 0 \\ 
 s_{21} & s_{22} \\ 
\end{array}
\right) 
\left( 
\begin{array}{cc} 
r_{11} & r_{12} \\ 
 0 & r_{22} \\ 
\end{array}
\right)  \\
&= \left( 
\begin{array}{cc} 
r_{11}s_{11}+r_{12}s_{21} & * \\ 
 r_{22}s_{21} & r_{22}s_{22} \\ 
\end{array}
\right) 
-\left( 
\begin{array}{cc} 
r_{11}s_{11} & * \\ 
 r_{11}s_{21} & r_{22}s_{22}+s_{21}r_{12} \\ 
\end{array}
\right)  \\
&= \left( 
\begin{array}{cc} 
r_{12}s_{21} & * \\ 
 (r_{22}-r_{11})s_{21} & -s_{21}r_{12} \\ 
\end{array}
\right).  
\end{align*}
Since $r\in\mathfrak{b}^{rss}$, $s_{21}=0$ and since $s=\diag(s)$, we are done.  

Now assume that the proposition holds when 
the rank of $r$ is $n$ (so $s$ is diagonal). 
Consider $r'$ and $s'$ where 
\[ r' = 
\left( \begin{array}{cc}
r & r_{\iota, n+1} \\ 
0 & r_{n+1,n+1} \\ 
\end{array}
\right) 
\mbox{ and }
s' = 
\left( \begin{array}{cc}
s & 0    \\ 
s_{n+1, \gamma} & s_{n+1, n+1} \\ 
\end{array}
\right),  
\] 
$(n+1)\times (n+1)$ matrices whose upper left blocks are $n\times n$ matrices $r$ and $s$, respectively, 
with $(r',s',i,j)\in \mu^{-1}(0)^{rss}$ and $i=0$ or $j=0$,
$i\in \mathbb{C}^{n+1}$, $j\in (\mathbb{C}^{n+1})^*$. 
Note that $(r_{\iota , n+1})$ is an $n \times 1$ matrix and 
$s_{n+1, \gamma}$ is a $1\times n$ matrix.  
Since  $\mu(r',s',i,j)=0$ with $i=0$ or $j=0$, we have 
$[r',s']=0$ in $\mathfrak{b}^*\subseteq \mathfrak{gl}_{n+1}^*$, 
and $r_{n+1, n+1}$ is distinct from $r_{ll}$ for all $l<n+1$  by assumption.   
So 
\[
([r',s'])_{\iota\gamma} 
=
 \left\{ 
\begin{aligned} 
  0   \;\;\;\;\;\;\;\;\;\;\;\;\;\;\;\;\;\;\;\;\;\;\;\;\;\;\;\;     & \mbox{ if $\iota< \gamma$},  \\ 
(r_{\iota\iota}-r_{\gamma\gamma}) s_{\iota\gamma} +\sum_{k>\iota} r_{\iota k}s_{k\gamma }-\sum_{k<\gamma}s_{\iota k}r_{k\gamma}  \:\:\: &  \mbox{ if $\iota\geq \gamma$}.  
\end{aligned} 
\right. 
\] 
Solely for notational purposes, let $m =n+1$. 
We rewrite the Lie bracket $[r',s']$ as the sum 
\[ 
\begin{aligned}
&\left( 
\begin{array}{cc}
[r,s] & 0 \\ 
0 & 0 \\  
\end{array}  
\right) +            \\ 
& 
\left( 
\begin{array}{ccccc}
r_{1m}s_{m 1} &   & &  & \\ 
r_{2 m}s_{m 1} & r_{2 m}s_{m 2} & & &   \\ 
r_{3 m}s_{m 1} & r_{3 m}s_{m 2} & r_{3 m}s_{m 3} & &  \\ 
\vdots & \vdots  & \vdots&   &   \\
 r_{\iota m}s_{m 1} & r_{\iota m}s_{m 2} & r_{\iota m}s_{m 3} &  \ddots  & \\ 
 \vdots & \vdots & \vdots & &   \\  
(r_{mm}-r_{11})s_{m 1} & (r_{mm}-r_{22})s_{m 2} -s_{m 1}r_{12} & 
(r_{mm}-r_{33})s_{m 3} 
-\displaystyle{\sum_{\iota =1}^{2}} s_{m \iota}r_{\iota 3} 
& \ldots & -\displaystyle{\sum_{k< m}} s_{m k}r_{k m} \\ 
\end{array}
\right).     \\ 
\end{aligned}
\] 
Consider the entry in the furthest bottom left corner of the large matrix. 
 Since the eigenvalues of $r'$ are pairwise distinct, $s_{m 1}=0$. 
 Moving 
 over one column to the right and remaining in the last row, we see that 
 $s_{m 2}=0$ as well. 
Continue recursively until we get to the last column in the large matrix:  
each term in the sum in the bottom right entry equals zero since 
each $s_{m k}=0$ for all $k<m$. 
So the large matrix on the right hand side of the sum equals zero, 
and it follows that 
$[r,s]=0$. 
The induction hypothesis states that if $(r,s,i', j')$ is in $\mu^{-1}(0)^{rss}$ where $r\in \mathfrak{b}^{rss}\subseteq \mathfrak{gl}_n$, 
  $s\in \mathfrak{b}^*\subseteq \mathfrak{gl}_n^*$, and $i'$, ${j'}^t\in \mathbb{C}^n$, with one of $i'$ or $j'$ equaling zero (${j'}^t$ is the transpose of $j'$), 
  then $s$ is diagonal.  
  Since $i'=0$ or $j'=0$, $\mu(r,s,i', j')=0$ is equivalent to $[r,s]=0$. 
We can now invoke the strong induction hypothesis to conclude 
$s=\diag(s)$ and by the above argument, $s'=\diag(s')$, which completes the proof. 
\end{proof}

\begin{lemma}\label{lemma:r-s-0-0-s-is-diagonal}
Let $(r,s,0,0)$ be in $\mu^{-1}(0)^{rss}$. Then there exists $b\in B$ such that $(\Ad_b(r), \Ad_b^*(s),0,0)$
 $=$ $(\diag(r),\diag(s),0,0)$. 
\end{lemma}

The Borel matrix in Lemma~\ref{lemma:r-s-0-0-s-is-diagonal} simultaneously diagonalizes $r$ and $s$. 
We add the subtlety in Lemma~\ref{lemma:r-s-0-0-s-is-diagonal} that the $B$-action on the points in $B.(r,s,0,0)$ does not change the diagonal coordinates of $r$ and $s$.

\begin{proof}
By Proposition~\ref{proposition:s-in-HR-is-diagonalizable}, $s$ is diagonal. By 
Proposition~\ref{proposition:distinct-ev-diagonalizable}, there exists a matrix $b$ in the Borel so that $brb^{-1}$ is diagonal. By the second statement after Lemma~\ref{lemma:s-diagonal-conjugation-still-diagonal}, 
we see that $\Ad_b^*(s)$ in $\mathfrak{b}^*$ is always  diagonal. By Proposition~\ref{proposition:diagonal-being-invariant} and Lemma~\ref{lemma:s-diagonal-conjugation-still-diagonal}, 
we see that the diagonal coordinates of $r$ and $s$ are not affected by the $B$-action.  
\end{proof}

\begin{proposition}\label{proposition:i-j-zero-closed-orbit}
Each $B$-orbit containing the quadruple $(r,s,0,0)$, where $r$ is regular semisimple and the commutator of $r$ and $s$ is zero, is closed.
\end{proposition}

\begin{proof}
By Proposition~\ref{proposition:s-in-HR-is-diagonalizable}, $s$ must be diagonal.  
Choose an appropriate 1-parameter subgroup $\lambda(t)$ so that 
\[ \lim_{t\rightarrow 0} \lambda(t).( r,\diag(s),0,0)=(\diag(r),\diag(s),0,0).  \]
Since $(\diag(r),\diag(s),0,0)$ is in the $B$-orbit by 
Lemma~\ref{lemma:r-s-0-0-s-is-diagonal}, we are done. 
\end{proof}

\begin{corollary}\label{corollary:Dreg-affine-quotient-closed-orbits}
The affine quotient $\{(r,s,0,0):r \mbox{ regular}, [r,s]=0 \}/\!\!/B$ consists of closed orbits. 
\end{corollary}

\begin{proof}
This follows from Proposition~\ref{proposition:i-j-zero-closed-orbit}. 
\end{proof}

\begin{corollary}\label{corollary:closed-orbits-are-disjoint}
If the points 
$(\diag(r),\diag(s),0,0)$ and $(\diag(r'),\diag(s'),0,0)$ are distinct, then the $B$-orbits 
$B.(\diag(r),\diag(s),0,0)$ and $B.(\diag(r'),\diag(s'),0,0)$ are disjoint. 
\end{corollary}

\begin{proof}
Suppose the $B$-orbits are not disjoint. 
We will show that the quadruples 
$(\diag(r),\diag(s),0,0)$ are $(\diag(r'),\diag(s'),0,0)$ are the same point in $T^*(\mathfrak{b}\times\mathbb{C}^n)$. 
By assumption, 
the intersection of the $B$-orbits is nonempty, so choose $b\in B$ so that $b.(\diag(r),\diag(s),0,0)=(\diag(r'),\diag(s'),0,0)$. 
This implies  
$(b\diag(r)b^{-1}, b\diag(s)b^{-1},0,0)$ $=$ $(\diag(r')$, $\diag(s')$, $0$, $0)$.  
Setting 
the first coordinates equal, 
$\Ad_b(\diag(r))=\diag(r')$. 
So $b\diag(r)b^{-1}$ is diagonal. By Proposition~\ref{proposition:diagonal-being-invariant}, 
$\diag(r)$ is fixed under the $B$-action: 
$\Ad_b(\diag(r)) = \diag(r)=\diag(r')$. 
Next, 
we see that 
$\Ad_b^*(\diag(s))=\diag(s)=\diag(s')$  
by Lemma~\ref{lemma:s-diagonal-conjugation-still-diagonal}. 
\end{proof} 

\begin{remark}\label{remark:exhausting-all-closed-orbits}
Proposition~\ref{proposition:openlocus-affine-quotient} together with Proposition~\ref{proposition:i-j-zero-closed-orbit}
show that all closed $B$-orbits in $\mu^{-1}(0)^{rss}$ 
must be of the form $(r,s,0,0)$.  
\end{remark}

\begin{proposition}\label{proposition:openlocus-affine-quotient}
Each orbit closure in $\mu^{-1}(0)^{rss}/\!\!/B$ contains a point of the form $(\diag(r),\diag(s'),0,0)$.  
\end{proposition}

\begin{proof}
For $(r,s,i,j)$ be in $\mu^{-1}(0)^{rss}$,  
we will show that there exists a quadruple of the form 
$(\diag(r)$, $\diag(s')$, $0$, $0)$ in its $B$-orbit closure, where $s'$ is some other element in $\mathfrak{b}^*$. 

Firstly, consider points of the form $(r,s,0,0)$. By Lemma~\ref{lemma:r-s-0-0-s-is-diagonal}, we are done.  
Thus consider points of the form $(r,s,i,j)$ in $\mu^{-1}(0)^{rss}$
with $i$ or $j$ not necessarily 0. 
By 
Proposition~\ref{proposition:distinct-ev-diagonalizable}, there is $b\in B$ so that 
$(\Ad_b(r),\Ad_b^*(s), bi, jb^{-1})=(\diag(r),s',i',j')$. 
Let us write $r= (r_{\iota\gamma})$, $s'=(s_{\iota\gamma}')$, $i'=(x_{\iota}')$, and $j'=(y_{\iota}')$.  
Since $\mu$ is $B$-equivariant, $[r,s]+ij=0$ implies $[\diag(r),s']+i'j'=0$ in $\mathfrak{b}^*$. 
Since 
\[
([\diag(r),s'] + i'j')_{\iota\gamma} = \left\{ 
\begin{aligned} 
0       \;\;\;\;\;\;\;\;\;\;\;\;\;\;\;\; & \mbox{ if $\iota< \gamma$}, \\ 
x_{\iota}'y_{\iota}' \;\;\;\;\;\;\;\;\;\;\;\; \;\; &  \mbox{ if $\iota=\gamma$}, \\ 
(r_{\iota\iota}-r_{\gamma\gamma})s_{\iota\gamma}' + x_{\iota}'y_{\gamma}' \;\; & \mbox{ if $\iota>\gamma$},   
\end{aligned} 
\right. 
\] 
  $x_{\iota}'=0$ or $y_{\iota}'=0$ for each $\iota$. 
At this point, we give a recipe for choosing the $a_{\iota}$'s in the 1-parameter subgroup $\lambda(t)=\diag(t^{a_1},\ldots,t^{a_n})$ where $a_{\iota}\in\mathbb{Z}$. Choose 
\[
a_{\iota}= \left\{ 
\begin{aligned} 
1   \;\;\;    & \mbox{ if $x_{\iota}' \not=0$}, \\ 
-1  \;\;\; &  \mbox{ if $y_{\iota}'   \not=0$}, \\ 
0   \;\;\; & \mbox{ if $x_{\iota}'=y_{\iota}'=0$}.  
\end{aligned} 
\right. 
\] 
Now consider $\lambda(t).(\diag(r),s',i',j')$, which equals 
\[\left( 
\diag(r), \left( 
\begin{array}{cccc}
s_{11} &                    &         &  \\ 
       & \ddots             &         &  \\  
       &  t^{a_{\iota}-a_{\gamma}}s_{\iota\gamma}'& \ddots  & \\ 
       &                    &         & s_{nn}\\
\end{array}
\right), 
\left( \begin{array}{c}
t^{a_1}x_1' \\ 
\vdots \\  
t^{a_n}x_n' \\  
\end{array}  
\right),  
\left(    
\begin{array}{ccc}
t^{-a_1} y_1' & \ldots & t^{-a_n} y_n' \\ 
\end{array}
\right)  
\right). 
 \]
If $x_{\iota}'\not=0$, then we have $t x_{\iota}'$. If $y_{\gamma}'\not=0$, then we also have $ty_{\gamma}'$. 
If $s_{\iota\gamma}'\not=0$, then since $s_{\iota\gamma}'=-x_{\iota}'y_{\gamma}'/(r_{\iota\iota}-r_{\gamma\gamma})$, we have $t^2 s_{\iota\gamma}'$. 
We conclude 
\[ \lim_{t\rightarrow 0} \lambda(t).(\diag(r),s',i',j')=(\diag(r),\diag(s'),0,0).
\] 
\end{proof}

The proof for Theorem~\ref{theorem:rss-extending-to-all-of-the-reg-semi-stable-locus} 
includes showing that 
the closure of each $B$-orbit on $\mu^{-1}(0)^{rss}/\!\!/B$ contains at most one point of the form $(\diag(r),\diag(s'),0,0)$.

\begin{remark}\label{remark:orbitclosures-emptyintersection-implies-orbitsempty-intersection}
It is clear that if $\overline{B.P_1}\cap\overline{B.P_2}=\varnothing$ where $P_1, P_2\in T^*(\mathfrak{b}\times \mathbb{C}^n)$, then $B.P_1\cap B.P_2=\varnothing$. 
\end{remark}

\section{Proof of Theorem  \texorpdfstring{$\ref{theorem:rss-extending-to-all-of-the-reg-semi-stable-locus}$}{rss categorical quotient}}\label{section:proof-of-reg-semistable-locus}

We now prove  Theorem~\ref{theorem:rss-extending-to-all-of-the-reg-semi-stable-locus}.

\begin{proof} 
It is clear that $P$ is regular and by Proposition~\ref{proposition:diagonal-being-invariant} and \ref{proposition:deriving-diag-coordinates-of-bsb-inverse}, 
 $P$ is $B$-invariant. 
By 
  Remark~\ref{remark:exhausting-all-closed-orbits}, closed orbits are precisely those that contain a point of the form $(r,s,0,0)$ since orbits that contain a point of the form $(r,s,i,j)$, where $i$ or $j$ is nonzero, are not closed. 
Since 
 each 
   closed orbit contains a unique point of the form $(\diag(r),\diag(s),0,0)$ by Lemma~\ref{lemma:r-s-0-0-s-is-diagonal} and by Corollary~\ref{corollary:closed-orbits-are-disjoint}, it is clear that two such distinct quadruples are mapped to distinct points in $\mathbb{C}^{2n}\setminus \Delta_n$ via the map $P$. In particular, we see that two  $B$-orbits $B.(\diag(r),\diag(s),0,0)$ and $B.(\diag(r'),\diag(s'),0,0)$ cannot be in the same $B$-orbit closure  
 for 
 if the intersection of the two orbits $B.(\diag(r),\diag(s),0,0)$ and $B.(\diag(r'),\diag(s'),0,0)$ were nonempty, then two such closed orbits would not be separated by $P$. 
Thus each $B$-orbit closure contains a unique closed orbit, and 
exactly one point of the form $(\diag(r),\diag(s),0,0)$. 
Thus $P$ is injective on orbit closures.

We will now show that $P$ is surjective. Consider $(x_{11},\ldots,x_{nn},y_{11},\ldots,y_{nn})$ in $\mathbb{C}^{2n}\setminus \Delta_n$. 
A point $(r,s,i,j)$ in $\mu^{-1}(0)^{rss}$ is constrained by $[r,s]+ij=0$. 
Since 
\[ 
([r,s]+ij)_{\iota\gamma} =
\left\{ 
\begin{aligned}
0 \;\;\;\;\;\;\;\; \;\;\;\;\;\; \;\;\;\; \;\;\;\;\;\;\;\;    &\mbox{ if }\iota < \gamma,  \\  
\sum_{\iota <k\leq n} r_{\iota k}s_{k \iota} 
- \sum_{1\leq k< \iota } s_{\iota k}r_{k \iota}+ x_{\iota}y_{\iota} \;\;\;\; &\mbox{ if } \iota = \gamma,  \\ 
\sum_{\iota \leq k\leq n} r_{\iota k}s_{k\gamma} 
- \sum_{1\leq k\leq \gamma}s_{\iota k}r_{k\gamma}+ x_{\iota} y_{\gamma}\;\;    &\mbox{ if }\iota >\gamma, \\ 
\end{aligned}
\right. 
\] 
none 
 of 
 the 
 coordinate functions of $[r,s]+ij$ involve $s_{\iota\iota}$'s. So $s_{\iota\iota}$ is a free parameter. 
Take $s_{\iota\iota}$ to equal 
\begin{equation}\label{eqn:setting-s-iota-iota-equal-to-an-expression}
 s_{\iota\iota} 
=  y_{\iota\iota}
- \tr(L^{\iota}(s-s_{\iota\iota}\I ))
 \end{equation}
and 
 take $r_{\iota\iota}$ to equal $x_{\iota\iota}$ (note that there are no $s_{\iota\iota}$ in the expansion of the expression on the right hand side of \eqref{eqn:setting-s-iota-iota-equal-to-an-expression}). 
Then a quadruple $(r,s,i,j)$ satisfying such conditions, whose $B$-orbit closure contains the unique point $(\diag(x_{\iota\iota}),\diag(y_{\iota\iota}),0,0)$, will map to $(x_{11},\ldots,x_{nn},y_{11},\ldots,y_{nn})$. 
\end{proof}

\section{The coordinate ring of the \texorpdfstring{$rss$}{rss}-locus}\label{section:coordinate-ring-rss-locus}

We will prove that 
the coordinate ring of the affine quotient $\mu^{-1}(0)^{rss}/\!\!/B$ is isomorphic to the coordinate ring of pairwise diagonal matrices in $\mu^{-1}(0)^{rss}$.

\subsection{Changing coordinates}\label{subsection:changing-coordinates}

\begin{definition}\label{definition:subdiagonal-level-k}
Let $(a_{\iota\gamma})$ be a matrix. {\em Level $k$ subdiagonal entries} consist of those coordinates $a_{\iota\gamma}$ below the main diagonal that satisfy $\iota-\gamma=k$. 
\end{definition}

\begin{example}\label{example:subdiagonal-level-k-nxn-matrix}
For an $n\times n$ matrix, level $0$ subdiagonal entries are precisely those along the main diagonal. Level $1$ subdiagonal entries are those immediately below the main diagonal. Level $n-1$ subdiagonal entry is the $(n,1)$-entry, in the lower left corner. 
\end{example}
 
In the next Proposition, we prove that for $\iota > \gamma$,  
the equation $(\mu(r,s,i,j))_{\iota\gamma}=0$
 may be solved for the coordinate function $s_{\iota\gamma}$, which depends on those $s_{ij}$ satisfying $i-j>\iota-\gamma$. That is,  each $s_{\iota\gamma}$ is a regular function of $s_{ij}$ in level $k$ subdiagonal, where $k=i-j>\iota-\gamma$. 

\begin{proposition}\label{proposition:off-diagonal-coordinate-function-defining-variety}
For each $\iota>\gamma$, the coordinate equation $([r,s]+ij)_{\iota\gamma}=0$ 
may be solved for $s_{\iota\gamma}$, which is in 
\[\im \left( \mathbb{C}[ 
\{r_{\iota j}\}_{\iota<j}, \{r_{i\gamma} \}_{i<\gamma},\{s_{ij}\}_{i-j>\iota-\gamma},
x_{\iota},y_{\gamma}][(r_{\iota\iota}-r_{\gamma\gamma})^{-1}]
\longrightarrow\mathbb{C}[\mu^{-1}(0)^{rss}]\right). 
\] 
\end{proposition}

\begin{proof}
For $\iota>\gamma$, the sequence of equalities 
\[  
0 =([r,s]+ij)_{\iota\gamma}= \sum_{\iota\leq j}r_{\iota j}s_{j \gamma}-\sum_{i\leq \gamma}s_{\iota i}r_{i\gamma}+x_{\iota}y_{\gamma} \\ 
=(r_{\iota\iota}-r_{\gamma\gamma} )s_{\iota\gamma} +\sum_{\iota <j}r_{\iota j}s_{j\gamma}-\sum_{i<\gamma}s_{\iota i}r_{i\gamma} + x_{\iota}y_{\gamma} 
\] 
 implies 
\[ 
   s_{\iota\gamma} =\dfrac{1}{r_{\gamma\gamma}-r_{\iota\iota}}\left( \sum_{\iota <j}r_{\iota j}s_{j\gamma}-\sum_{i<\gamma}s_{\iota i}r_{i\gamma} + x_{\iota}y_{\gamma}\right).   
\] 
\end{proof}

\begin{coordinatering}\label{coordinatering:coordinate-ring-off-diagonal}
We apply Proposition~\ref{proposition:off-diagonal-coordinate-function-defining-variety}  starting 
  from level $n-1$ subdiagonal of 
  $[r,s]+ij=0$, 
    add $(r_{nn}-r_{11})^{-1}$ to, and thus will be able to remove the parameter $s_{n1}$ from, the coordinate ring $\mathbb{C}[\mu^{-1}(0)^{rss}]$. 
     We then move to level $n-2$ subdiagonal and repeat the procedure by adding 
$(r_{n-1,n-1}-r_{11})^{-1}$ to the ring and then removing $s_{n-1,1}$, and then 
adding  
 $(r_{nn}-r_{22})^{-1}$ to the coordinate ring and then removing 
 $s_{n2}$. 
  Continue by moving up to the next subdiagonal.  
 \end{coordinatering}

\begin{corollary}\label{corollary:off-diag-coord-fn-defining-variety-iterated-steps}
Applying Systematic Procedure~\ref{coordinatering:coordinate-ring-off-diagonal}, 
it follows from Proposition~\ref{proposition:off-diagonal-coordinate-function-defining-variety} that $s_{\iota\gamma}$ is in 
\[ \begin{aligned} 
\im (\: \mathbb{C}[\{r_{ij}\}_{j>i\geq \iota\mbox{ or }i<j\leq \gamma}, 
&\{ x_k\}_{k\geq \iota},\{ y_l\}_{l\leq \gamma}][\{(r_{ii}-r_{jj})^{-1} \}_{i-j\geq \iota-\gamma} ]  
 \rightarrow \mathbb{C}[\mu^{-1}(0)^{rss}] \: ). \\
\end{aligned}
\] 
\end{corollary}

Corollary~\ref{corollary:off-diag-coord-fn-defining-variety-iterated-steps} shows each $s_{\iota\gamma}$ (where $\iota>\gamma$) may be solved for so that it does not depend on any of the entries of $s\in\mathfrak{b}^*$. 

\begin{proof}
 Exhaust Systematic Procedure~\ref{coordinatering:coordinate-ring-off-diagonal} 
 recursively by decreasing to the next sublevel (and thus moving closer to the main diagonal). 
\end{proof}

\begin{corollary}\label{corollary:diagonalfunction-matrix-variety}
After replacing 
 each parameter $s_{\mu\nu}$ in the coordinate function $([r,s]+ij)_{\iota\iota}$ by recursively applying Systematic Procedure~\ref{coordinatering:coordinate-ring-off-diagonal}, 
 we obtain that $([r,s]+ij)_{\iota\iota}$ is in the image 
\[ 
   \im(\mathbb{C}[r_{ij},\{ x_k\}_{k\geq \iota} , \{ y_l\}_{l\leq \iota} ] [(r_{ii}-r_{jj})^{-1}]\rightarrow \mathbb{C}[T^*(\mathfrak{b}\times \mathbb{C}^n)^{rss}]). 
\] 
\end{corollary}

\begin{proof}
This follows from Corollary~\ref{corollary:off-diag-coord-fn-defining-variety-iterated-steps} since for each $\iota > \gamma$,  $s_{\iota\gamma}$ may be solved 
so that it is independent of 
the coordinates of $s$, 
and also since 
$s_{\iota\iota}$ are not constrained under the moment map. 
\end{proof}

\begin{corollary}\label{corollary:diag-fn-defining-matrixvariety-explicit-eqn}
Writing 
  $F_{\iota}:=(\mu(r,s,i,j))_{\iota\iota}$, 
the image under the map given in Corollary~\ref{corollary:diagonalfunction-matrix-variety}
is  
\begin{align*}
F_{\iota} = \tr( jL^{\iota}i). 
 \end{align*}
\end{corollary}

\subsection{B-invariant functions}\label{subsection:B-invariantfunctions}

We will show that $B$-invariant functions on $\mu^{-1}(0)^{rss}$ include $F_{\iota}$ (involving $r$, $i$, and $j$) as  in Proposition~\ref{proposition:diagonal-rss-natural-matrix-representation}, $G_{\iota}$ (involving $r$ and $s$) as in Proposition~\ref{proposition:B-invariant-functions-involving-s}, $H_{\iota}$
(involving $r$) as in Proposition~\ref{proposition:B-invariant-functions-diagonal-entries-of-r}, and $K_{\gamma\nu}$ (involving the inverse of the difference of the diagonal coordinates of $r$) as in Proposition~\ref{proposition:B-invariant-invert-difference-diag-r}. They are summarized as follows: 
\[ 
\begin{aligned}
F_{\iota}(r,s,i,j) &=  \tr\left(j L^{\iota} i  \right), 
			\\
G_{\iota}(r,s,i,j)&=
  \tr \left( L^{\iota} s\right),  
			\\ 
H_{\iota}(r,s,i,j)&=\tr(L^{\iota}r),  
			\\
K_{\gamma\nu}(r,s,i,j) &= [\tr((L^{\nu}-L^{\gamma})r)]^{-1},  
			\\
\end{aligned}
\]  
where $1\leq \iota\leq n\mbox{ and }\: 1\leq \gamma< \nu \leq n$.

The rational functions above coincide with the classical notion that the trace of an oriented cycle (of a quiver) as well as the trace of a path that begin and end at a framed vertex is an invariant function (\cite{MR958897}, \cite{MR1834739}, \cite{MR1623674}). 
Furthermore, a strategy to calculate semi-invariant polynomials is given in the proof of Proposition 8.2.1 in the Appendix by Gan and Ginzburg (\cite{MR2210660}); 
these techniques apply to Nakajima's affine varieties $\mathfrak{M}_0(\mathbf{v},\mathbf{w})$ 
and quiver varieties $\mathfrak{M}(\mathbf{v},\mathbf{w})$ (\cite{MR1604167}).

\begin{proposition}\label{proposition:diagonal-rss-natural-matrix-representation}
Denoting 
$F_{\iota}(r,s,i,j):= ([r,s]+ij)_{\iota\iota}$ from Corollary~\ref{corollary:diag-fn-defining-matrixvariety-explicit-eqn},  
\[ 
F_{\iota}(r,s,i,j) =  \left[
\tr\left( \prod_{
1\leq k\leq n, k\not=\iota 
} l_k(r) \right)\right]^{-1}  \tr\left(  j \left(\prod_{
1\leq k\leq n, k\not=\iota
} l_k(r) \right)i \right),  
\]  a $B$-invariant function. 
\end{proposition}

\begin{proof}
For any $d\in B$, 
\[\begin{aligned} 
F_{\iota}(d.(r,s,i,j)) &=  \tr(jd^{-1} L^{\iota}(\Ad_d(r)) di) \\
&= \tr(  j d^{-1} \: d L^{\iota} (r)  d^{-1} \:  d i)   \mbox{ by Lemma~\ref{lemma:Baction-on-Liota-operator}} \\
&=\tr(  jL^{\iota}(r) i) = F_{\iota}(r,s,i,j).  \\
\end{aligned} 
\] 
\end{proof}

\begin{proposition}\label{proposition:B-invariant-functions-involving-s}
Denoting $G_{\iota}(r,s,i,j):=s_{\iota\iota}'$ from  \eqref{eqn:s-single-prime-first-prop},   
\[ G_{\iota}(r,s,i,j)=
\left[ \tr\left( \prod_{
1\leq k\leq n, k\not=\iota 
} l_k(r) \right)\right]^{-1} 
  \tr \left(\prod_{ 
1\leq k\leq n, k\not=\iota 
} l_k(r) \; s\right),   
\] 
a $B$-invariant function. 
\end{proposition} 

\begin{proof}
 For any $d\in B$, 
\[\begin{aligned} 
G_{\iota}(d.(r,s,i,j)) &= \tr(L^{\iota}(\Ad_d(r)) dsd^{-1}) \\
 &= \tr(d L^{\iota}(r) \cancel{d^{-1}}  \cancel{d}sd^{-1}) \mbox{ by Lemma~\ref{lemma:Baction-on-Liota-operator}}  \\ 
 &= \tr(d L^{\iota}(r)  sd^{-1}) \\ 
  &= \tr(L^{\iota}(r)  s) = G_{\iota}(r,s,i,j).\\  
\end{aligned} 
\] 
\end{proof}

\begin{proposition}\label{proposition:B-invariant-functions-diagonal-entries-of-r}
Denoting $H_{\iota}(r,s,i,j):=r_{\iota\iota}$ from Proposition~\ref{proposition:diag-entries-r-is-trace}, 
\[ H_{\iota}(r,s,i,j)=
\left[ \tr\left( \prod_{
1\leq k\leq n, k\not=\iota 
} l_k(r) \right)\right]^{-1} 
  \tr \left(\prod_{ 
1\leq k\leq n, k\not=\iota 
} l_k(r) \; r\right),  
\] a $B$-invariant function. 
\end{proposition}

\begin{proof}
This proof is analogous to the proof of  Proposition~\ref{proposition:B-invariant-functions-involving-s}: to prove this proposition, replace $s$ with $r$.  
\end{proof}   
 
\begin{corollary}\label{corollary:B-invar-fns-operator-and-r}
$H_{\iota}(r,s,i,j)$ in Proposition~\ref{proposition:B-invariant-functions-diagonal-entries-of-r} may be written as $e_{\iota}^* \: r\: e_{\iota}$, 
 where $e_{\iota}$ is the standard basis vector for $\mathbb{C}^n$ and $e_{\iota}^*$ is the standard covector. 
\end{corollary}

\begin{proof}  
This is clear, and the product $e_{\iota}^* \: r\: e_{\iota}$ of matrices is $B$-invariant since for any $d\in B$, $(drd^{-1})_{\iota\iota}=d_{\iota\iota}r_{\iota\iota}d_{\iota\iota}^{-1}=r_{\iota\iota}$. 
\end{proof} 

\begin{proposition}\label{proposition:B-invariant-invert-difference-diag-r}
Denoting $K_{\gamma\nu}(r,s,i,j):= (r_{\nu\nu} - r_{\gamma\gamma})^{-1}$ 
from Proposition~\ref{proposition:difference-of-r-is-trace},  
where 
$1\leq \gamma < \nu \leq n$,  
\[ K_{\gamma\nu}(r,s,i,j) = [\tr((L^{\nu}-L^{\gamma})r)]^{-1}, 
\] a $B$-invariant function.  
\end{proposition} 

\begin{proof} 
For any $1\leq \gamma< \nu\leq n$,  
$\left( K_{\gamma\nu}(r,s,i,j)\right)^{-1} =  \tr(L^{\nu}r - L^{\gamma}r)  
=  \tr(L^{\nu}r)  -  \tr(L^{\gamma}r)$. So $(K_{\gamma\nu})^{-1}$ is $B$-invariant. Since it is never vanishing, 
$K_{\gamma\nu}$ is $B$-invariant. 
\end{proof} 

\begin{corollary}\label{corollary:B-invar-fns-difference-of-operator-and-r} 
 $K_{\gamma\nu}(r,s,i,j)$ in Proposition~\ref{proposition:B-invariant-invert-difference-diag-r} 
 may be written as 
$(e_{\nu}^* (r-r_{\gamma\gamma}\I)e_{\nu})^{-1}$. 
\end{corollary}  

\begin{proof}
This is clear, and the product of matrices 
$(e_{\nu}^* (r-r_{\gamma\gamma}\I)e_{\nu})^{-1}$ is $B$-invariant since for any $d\in B$ and for any $\gamma < \nu$, 
$(d(r-r_{\gamma\gamma}\I)d^{-1})_{\nu\nu} 
 =(drd^{-1}-r_{\gamma\gamma}\I)_{\nu\nu} 
 = (drd^{-1})_{\nu\nu} -(r_{\gamma\gamma}\I)_{\nu\nu}
=r_{\nu\nu}-r_{\gamma\gamma}$, which is never vanishing; so its inverse is well-defined. 
\end{proof}

\subsection{The initial ideal and regular sequence}\label{subsection:initial-ideal-regular-sequence}

\begin{definition}\label{definition:passingfrom-regular-to-polynomial-concept}
We will define $z_{ij}^{(kl)}:= \dfrac{r_{kl}}{r_{ii}-r_{jj}} $ for $i\not= j$. 
\end{definition}


\begin{weight}\label{weight:weight-on-coordinate-ring} 
The coordinate ring $\mathbb{C}[x_k,y_l,z_{ij}^{(kl)}]$ has the following integral weight on each variable: $\wt(x_k)= 1, \wt(y_l)=1, \wt(z_{ij}^{(kl)})=0$. 
\end{weight} 

\begin{monomial}\label{monomial:grobner-basis-monomial-ordering}
We fix a term order $>$ on $\mathbb{C}[x_k, y_l, z_{ij}^{(kl)}]$ via the following refinement: 
write 
\[ m= x_1^{a_1}\cdots x_n^{a_n}
y_n^{b_n}\cdots y_1^{b_1}
(z_{ij}^{(kl)})^{c_{ij}^{(kl)}} >_{\lex,\rev}
 x_1^{a_1'}\cdots x_n^{a_n'} 
 y_n^{b_n'}\cdots y_1^{b_1'}
 (z_{ij}^{(kl)})^{{c_{ij}^{(kl)}}'}=m'
\] 
  if
\[ (a_1,\ldots, a_n,b_n,\ldots, b_1,c_{ij}^{(kl)}) - 
  (a_1',\ldots, a_n', b_n',\ldots, b_1', c_{ij}^{(kl)'})>0 
\] in the sense that the left-most nonzero coordinate of the difference of the exponent vectors is positive. 

We impose any ordering on the $z_{ij}^{(kl)}$ as long as they succeed the ordering on the $x_{\iota}$'s and the $y_{\gamma}$'s; thus, we will view them as constants, which coincide with their weights as imposed in Remark~\ref{weight:weight-on-coordinate-ring}. 
\end{monomial}	

We will write $>$ rather than $>_{\lex,\rev}$ throughout this section. 

\begin{remark}\label{remark:monomial-ordering-no-mention-of-total-ordering} 
Total ordering by total degree in Monomial Ordering~\ref{monomial:grobner-basis-monomial-ordering} does not need to be mentioned since each $F_{\iota}$ is a homogeneous quadratic function. Furthermore, if we want to view $z_{ij}^{(kl)}$ as a rational function (rather than as a constant) and impose an ordering, one may define such ordering by $\dfrac{f_1}{f_2}>\dfrac{g_1}{g_2}$ if $f_1 g_2 > f_2 g_1$. 
%
%
\end{remark}

Note that we have imposed lexicographical order on the $x_i$'s, and 
reversed the indices on the $y_i$'s and applied lex on the $y_i$'s
(caution: this is not the same as reverse lex order since that has infinite descending sequences; thus it is not a monomial order), with the ordering on the $x_i$'s preceding the $y_i$'s. Monomial Ordering~\ref{monomial:grobner-basis-monomial-ordering} of  $\mathbb{C}[x_k,y_l,z_{ij}^{(kl)}]$ is multiplicative (i.e., if $m>m'$, then $m\widetilde{m}> m'\widetilde{m}$) and artinian ($m>1$ for all nonunit monomials $m$).

\begin{lemma}\label{lemma:initial-term-of-each-F}
With respect to Monomial Ordering~\ref{monomial:grobner-basis-monomial-ordering}, the initial term $\In(F_{\iota})$ of each $F_{\iota}$ equals $x_{\iota}y_{\iota}$.
\end{lemma}

\begin{proof} 
Since each monomial corresponds to a unique exponent vector, 
write the exponents of each monomial of $F_{\iota}$ as a pair of multi-indices $\mathbf{a}=(a_1,\ldots, a_n)$ and $\mathbf{b}=(b_1,\ldots, b_n)$, i.e., 
$\mathbf{a}$ and $\mathbf{b}$ 
may be thought of as column vectors living in $\mathbb{Z}_{\geq 0}^n$ (for the time being, we omit keeping track of $z_{ij}^{(kl)}$). 
It is clear by Corollary~\ref{corollary:diag-fn-defining-matrixvariety-explicit-eqn} that 
both $\mathbf{a}$ and $\mathbf{b}$ are in 
$\{e_1,\ldots, e_n \}$, where $e_i$ is the standard basis vector for $\mathbb{Z}^n$, since exactly one of the exponents for $x_{\gamma}$'s and one of the exponents for  $y_{\nu}$'s are nonzero for each monomial of $F_{\iota}$. Since the degree of each monomial of $F_{\iota}$ is 2, higher powers of $x_{\gamma}$ or $y_{\nu}$ cannot occur. 

Now for a fixed $F_{\iota}$, we inspect the monomials in Corollary~\ref{corollary:diag-fn-defining-matrixvariety-explicit-eqn} to conclude 
the inclusion of sets 
\[\{ \mathbf{a}: \mathbf{a} \mbox{ is the multi-index of some monomial of }F_{\iota} \}\supseteq \{e_{\iota},\ldots, e_{n} \}.  
\] 
 The vector $\mathbf{a}$ corresponding to the first term $x_{\iota}y_{\iota}$ is $e_{\iota}$ while all the other summations show that $\mathbf{a}$ is in $\{ e_{\iota+1},\ldots, e_n\}$.   
When $\mathbf{a} =e_{\iota}$, 
 the possibilities for its corresponding  $\mathbf{b}$-vector 
 take on all values $e_{1},\ldots, e_{\iota}$,  
 which one may check by looking at the monomials in 
 Corollary~\ref{corollary:diag-fn-defining-matrixvariety-explicit-eqn}.  
In order to determine the leading term, 
if $\mathbf{a}=e_{\alpha}$, we want $\alpha$ to be as small as possible 
 since we have imposed lex on the $x_{\gamma}$'s, and  
 if $\mathbf{b}=e_{\beta}$, we want $\beta$ to be as big as possible  
since we have imposed a reverse ordering on the $y_{\nu}$'s. 
Since $x_{\iota}y_{\iota}$ occurs once in $F_{\iota}$ with coefficient 1 with 
$c_{ij}^{(kl)}=0$, the initial term of $F_{\iota}$ is $x_{\iota}y_{\iota}$.  
\end{proof}


\begin{lemma}\label{lemma:initial-term-of-Fiota}
The initial terms of $\{ F_{\iota}\}_{1\leq \iota \leq n}$ form a regular sequence. 
\end{lemma} 

\begin{proof}
This follows from Lemma~\ref{lemma:initial-term-of-each-F}.  
\end{proof} 

\begin{lemma}\label{lemma:deriving-reg-seq-from-initialterms} 
The set $\{ F_{\iota}\}_{1\leq \iota\leq n}$ of functions is $\mathbb{C}[T^*(\mathfrak{b}\times \mathbb{C}^n)^{rss}]$-regular. 
\end{lemma}  

\begin{proof} 
The $F_{\iota}$'s form a regular sequence since their initial terms form a regular sequence by Proposition 15.15 in \cite{MR1322960}.  
\end{proof}

\section{Proof of Theorem \texorpdfstring{$\ref{theorem:rss-locus-to-complex-space-B-invariant}$}{rss isomorphism of varieties}}\label{section:proof-of-last-proposition-coordinate-ring}

The following proves Theorem~\ref{theorem:rss-locus-to-complex-space-B-invariant}. 

\begin{proof}
We have 
\[ 
\begin{aligned}
\mathbb{C}[\mu^{-1}(0)^{rss}] &= \mathbb{C}[T^*(\mathfrak{b}\times\mathbb{C}^n)^{rss}
]/\! \left<\: (\mu(r,s,i,j))_{\iota\gamma} \:\right>  \\ 
&\cong \dfrac{\mathbb{C}[r_{\alpha\beta}, \: x_k, \: y_l ][(r_{\nu\nu}-r_{\gamma\gamma})^{-1}]}{
\left<\: ([r,s]+ij)_{\iota\iota} \:\right> } \otimes \mathbb{C}[s_{11},\ldots, s_{nn}]  \\ 
&= \dfrac{\mathbb{C}[r_{\alpha\beta}, \: x_k, \: y_l ][(r_{\nu\nu}-r_{\gamma\gamma})^{-1}]}{
\left<\: F_{\iota}(r,s,i,j) \: \right> } \otimes \mathbb{C}[s_{11},\ldots, s_{nn}],  
\end{aligned}
\] 
where the second isomorphism holds by Corollary~\ref{corollary:diagonalfunction-matrix-variety} and the third 
equality holds by Corollary~\ref{corollary:diag-fn-defining-matrixvariety-explicit-eqn}. 
The locus 
		$\mu^{-1}(0)^{rss}$ is a complete intersection by Lemma~\ref{lemma:deriving-reg-seq-from-initialterms}. 
By Propositions~\ref{proposition:diagonal-rss-natural-matrix-representation}, \ref{proposition:B-invariant-functions-involving-s},  \ref{proposition:B-invariant-functions-diagonal-entries-of-r}, and \ref{proposition:B-invariant-invert-difference-diag-r}, 
\[ \begin{aligned} 
\mathbb{C}[\mu^{-1}(0)^{rss}]^B &=
   \dfrac{ \mathbb{C}\left[
    F_{\iota}(r,s,i,j),  G_{\iota}(r,s,i,j), H_{\iota}(r,s,i,j)\right]
    \left[ K_{\gamma\nu}(r,s,i,j) 
 \right]}{
\left< F_{\iota}(r,s,i,j)\right> } \\
&\cong  \mathbb{C}\left[r_{11},\ldots, r_{nn}, s_{11}',\ldots,s_{nn}'  
\right]\left[(r_{\nu\nu}-r_{\gamma\gamma})^{-1} \right] \\   
&\cong \mathbb{C}[\mathbb{C}^{2n}\setminus \Delta_n]. 
\end{aligned}
\] 
\end{proof}


\appendix
 
\def\cprime{$'$} \def\cprime{$'$} \def\cprime{$'$} \def\cprime{$'$}
\providecommand{\bysame}{\leavevmode\hbox to3em{\hrulefill}\thinspace}
\providecommand{\MR}{\relax\ifhmode\unskip\space\fi MR }
\providecommand{\MRhref}[2]{%
  \href{http://www.ams.org/mathscinet-getitem?mr=#1}{#2}
}
\providecommand{\href}[2]{#2}

\end{document}